\documentclass[11pt,leqeq]{article}
\usepackage{amssymb,amsthm}
\usepackage{amsmath}
\usepackage{amscd}
\usepackage{xy}
\xyoption{all}
\usepackage[dvips]{graphicx}
\newtheorem{thm}{Theorem}
\newtheorem{lem}[thm]{Lemma}
\newtheorem{cor}[thm]{Corollary}
\newtheorem{prop}[thm]{Proposition}

\newtheorem{defn}[thm]{Definition}
\newtheorem{exmp}[thm]{Example}

\date{}

\begin{document}
\setlength{\baselineskip}{16pt}
\title{Quantum Boolean Algebras }
\author{Rafael D\'\i az} \maketitle

\begin{abstract}
We introduce quantum Boolean algebras which are the analogue of the Weyl algebras for
Boolean affine spaces. We study quantum Boolean algebras from the logical and set theoretical viewpoints.
\end{abstract}

\section{Introduction}

After Stone \cite{sto} and  Zhegalkin \cite{zh},
Boole's main contribution to science \cite{boo} can be understood as the realization that
the mathematics of logical phenomena is controlled -- to a large extend -- by the field $\mathbb{Z}_2=\{0,1\}$  with two elements; in contrast
the mathematics of classical physical phenomena is controlled -- to a large extend --  by  the field $\mathbb{R}$ of real numbers.
The switch from $\mathbb{Z}_2$ to $\mathbb{R}$  corresponds with a deep ontological jump from
logical to physical phenomena. The switch from $\mathbb{R}$ to  $\mathbb{C}$ corresponds to the jump from classical
to quantum physics.\\

What makes the logic/physics jump possible is the fact
that $\mathbb{Z}_2$ may be regarded as an object of two different categories.
On the one hand, it is a field $(\mathbb{Z}_2, +, . ) $ with sum and product defined by making $0$ the neutral element and $1$ the product unit.
On the other hand, it is a set of truth values  with
$0$ and $1$ representing falsity and truth, respectively. Indeed, $(\mathbb{Z}_2,\vee, \wedge, \overline{( \ )}) $ is a Boolean algebra: a complemented distributive lattice with minimum $0$ and maximum $1$. The operations $\vee, \wedge,$ and $\overline{( \ )}$ correspond
with the logical connectives $\mathrm{OR}$, $\mathrm{AND}$, and $\mathrm{NOT}.$
The two viewpoints are related by the identities:
$a \vee b = a+ b + ab,   \ a \wedge b = ab,    \ \overline{a}= a+1.$
These identities, together with the inverse relation $a+b=  (a  \wedge \overline{b}) \vee (\overline{a} \wedge b)$,  allow us to switch back and forth from the algebraic to the logical viewpoint.\\

By and large, the logical and algebraic viewpoints have remained separated.  In this work, in order to explore quantum-like phenomena in characteristic $2$, we place ourselves at the jump. Our algebraic viewpoint is, in a sense, complementary to the quantum logic approach  initiated by Birkhoff and von Neumann \cite{bi} based on the theory of lattices.
For example, while the meet in quantum logic is a commutative connective, we propose in this work a quantum analogue for the meet which turns out to be non-commutative.
The appearance of non-commutative operations is an essential feature of quantum mechanics \cite{ba, co1, von}.\\

We take as our guide the well-known fact that the quantization of canonical phase space may be identified with
the algebra of differential operators on configuration space. In analogy with the real/complex case, we introduce
 the algebra $\mathrm{BDO}_n$ of Boolean differential operators on $ \mathbb{Z}_2^n.$
We provide a couple of presentations by generators and relations of $\mathrm{BDO}_n$,  giving
rise to the Boole-Weyl algebras $\mathrm{BW}_n$ and the shifted Boole-Weyl algebras $\mathrm{SBW}_n$. We call these algebras the
quantum Boolean algebras. We study the structural coefficients of $\mathrm{BW}_n$ and $\mathrm{SBW}_n$ in various bases. \\

Having introduced quantum Boolean algebras, we proceed to study them from the logical and
set theoretical  viewpoints. For us, the main difference between classical
and quantum logic rest on the fact that classical observations, propositions, can be measured without, in principle, modifying
the state of the system; quantum observations, in contrast, are quantum operators:  the measuring process
changes the state of the system. Indeed, regardless of the actual state of the system, after measurement the system will be
an eigenstate of the observable. Quantum observables are operators acting on the states of the system, and thus quite different to
classical observables which are descriptions of the state of the system.\\

This work is organized as follows. In Section \ref{s1} we review some standard facts on regular functions on
affine spaces over $\mathbb{Z}_2$. In Section \ref{s2} we introduce $\mathrm{BDO}_n$, the algebra of Boolean differential operators
on $\mathbb{Z}_2^n$.
In Section \ref{bwa} we introduce the Boole-Weyl algebra  $\mathrm{BW}_n$ which is a presentation by generators
and relations of $\mathrm{BDO}_n$. We describe the structural coefficients of $\mathrm{BW}_n$ in several bases.
In Section \ref{s3} we introduce the shifted Boole-Weyl algebra  $\mathrm{SBW}_n$ which is another presentation by generators
and relations of $\mathrm{BDO}_n$, and describe the structural coefficients of $\mathrm{SBW}_n$ in several bases.
In Section \ref{s5} we discuss the logical aspects of our constructions: we introduce a quantum operational logic that generalizes classical propositional logic, and for which Boolean differential operators play a semantic role akin to that played by truth functions  in classical propositional logic. We use the theory of operads and props to describe our results. In Section \ref{s6} we adopt a set theoretical viewpoint and show that
just as classical propositional logic is intimately related with $\mathrm{P} \mathrm{P}(x)$, the Boolean algebra of sets of subsets
of $x$, quantum operational logic  is intimately related with $\mathrm{P} \mathrm{P}(x \sqcup x)$
the quantum Boolean algebra of sets of subsets of two disjoint copies of $x$. In the final Section \ref{s7} we make some closing remarks
and mention a few topics for future research.

\section{Regular Functions on Boolean Affine Spaces}\label{s1}

Our main goal in this work is to study the Boolean analogue for the Weyl algebras, and to describe those algebras
from a logical and a set theoretical viewpoints. Fixing a field $k$, the Weyl algebra $ \mathrm{W}_n $ over $k$ can be identified with the $k$-algebra of algebraic differential operators on the affine space $\mathbb{A}^n(k) = k^n.$  By definition \cite{h, s} the $k$-algebra $k[\mathbb{A}^n]$ of regular  functions on $k^n$ is the $k$-algebra of maps $$f:k^n \ \longrightarrow \ k$$ such that there exists a polynomial $F \in k[x_1,...,x_n]$ with  $f(a)=F(a)$ for all $a \in k^n.$ If $k$ is a field of characteristic zero, then the $k$-algebra of regular functions on
$k^n$ can be identified with $k[x_1,...,x_n]$ the polynomial  ring  of over $k$. Let $\partial_1,...,\partial_n$ be the derivations
of $k[x_1,...,x_n]$ given by $\partial_i x_j = \delta_{i,j}$ for $i,j \in [n]=\{1,...,n \}.$ The $k$-algebra $\mathrm{DO}_n$ of differential operators
on $k^n$ is the subalgebra of $$\mathrm{End}_k(k[x_1,...,x_n])$$ generated by $\partial_i$ and the operators
of multiplication by $\ x_i \ $ for $i \in [n]$. \\

By definition, the Weyl algebra $\mathrm{W}_n$ is the $k$-algebra defined via generators and relations as
$$k<x_1,...,x_n, y_1,..., y_n>/<x_ix_j - x_jx_i, \ y_iy_j - y_jy_i, \ y_ix_j - x_jy_i, \ y_ix_i - x_iy_i -1> ,$$
where $k<x_1,...,x_n, y_1,..., y_n>$ is the free associative $k$-algebra generated by $x_1,...,x_n, y_1,..., y_n,$ and
$<x_ix_j - x_jx_i, \ y_iy_j - y_jy_i, \ y_ix_j - x_jy_i , \ y_ix_i - x_iy_i -1>$
is the ideal generated by the relations $x_ix_j =x_jx_i$ and $y_iy_j = y_jy_i$ for $i,j \in [n]$, $y_ix_j = x_jy_i$  for $i \neq j \in [n]$, $y_ix_i = x_iy_i +1$ for $i \in [n]$.\\

The Weyl algebra $\mathrm{W}_n$ comes  with a natural representation $\mathrm{W}_n \ \longrightarrow \ \mathrm{End}_k(k[x_1,...,x_n])$
sending $y_i$ to $\partial_i$ and $x_i$ to the
operator of multiplication by $x_i$. This representation induces an isomorphism of algebras
$\mathrm{W}_n \rightarrow \mathrm{DO}_n.$\\

We proceed to study the analogue of the Weyl algebras for the Boolean affine spaces $\mathbb{A}^n(\mathbb{Z}_2)=\mathbb{Z}_2^n.$
First, we review some basic facts on regular functions on $\mathbb{Z}_2^n.$
Let $\mathrm{M}(\mathbb{Z}_2^n, \mathbb{Z}_2)$ be the $\mathbb{Z}_2$-algebra of all maps from $\mathbb{Z}_2^n$ to $\mathbb{Z}_2$ with pointwise addition
and multiplication. The $\mathbb{Z}_2$-algebra  $\mathbb{Z}_2[\mathbb{A}^n]$ of regular functions on $\mathbb{Z}_2^n$ is the
sub-algebra of $\mathrm{M}(\mathbb{Z}_2^n, \mathbb{Z}_2)$ consisting of the maps $f: \mathbb{Z}_2^n \longrightarrow \mathbb{Z}_2$ for which there exists a polynomial $F \in \mathbb{Z}_2[x_1,...,x_n]$
such that $f(a)=F(a)$ for all $a \in \mathbb{Z}_2^n.$
In this case $\mathbb{Z}_2[\mathbb{A}^n]$ is not a polynomial ring; instead we have the  following result.

\begin{lem}\label{bag}{\em There is an exact sequence of $\mathbb{Z}_2$-algebras
$$0 \ \longrightarrow \ <x_1^2 +x_1,...., x_n^2 +x_n> \ \longrightarrow \ \mathbb{Z}_2[x_1,...,x_n] \ \longrightarrow \   \mathbb{Z}_2[\mathbb{A}^n] \longrightarrow 0$$
where $<x_1^2 +x_1,...., x_n^2 +x_n>$ is the ideal generated by the relations $x_i^2=x_i$ for $i\in [n].$}
\end{lem}
 \begin{proof}
The map $\mathbb{Z}_2[x_1,...,x_n] \longrightarrow  \mathbb{Z}_2[\mathbb{A}^n]$ sends a polynomial $P$ to the
map $p:\mathbb{Z}_2^n \rightarrow \mathbb{Z}_2$ given by $p(a)=P(a).$ Clearly  $x_i^2$ and $x_i$ define
the same map $\mathbb{Z}_2^n \longrightarrow \mathbb{Z}_2$. Thus $<x_1^2 +x_1,..., x_n^2 +x_n>$ is in the kernel of the map $\mathbb{Z}_2[x_1,...,x_n] \longrightarrow  \mathbb{Z}_2[\mathbb{A}^n]$. Let $P \in \mathbb{Z}_2[x_1,...,x_n]$ be such that $P(a)=0$ for all $a \in \mathbb{Z}_2^n$.
We can write $P=(x_1^2 + x_1)Q + R,$ where $R \in \mathbb{Z}_2[x_1,...,x_n]$ is a degree $1$ polynomial in $x_1$. Therefore
$R = x_1S + T$ where $S,T \in \mathbb{Z}_2[x_2,...,x_n].$ Since $R(a)=0$ for all $a \in \mathbb{Z}_2^n$, then
$T(b)=R(0,b)=0$ for all $b \in \mathbb{Z}_2^{n-1}$, and thus we obtain that $S(b)=S(b)+T(b)=R(1,b)=0.$ We have shown that the polynomials
$S$ and $T$ define the $0$ function  and thus, by induction, we conclude that $S, T \in \ <x_2^2 +x_2,...., x_n^2 +x_n>$, which implies that $R$ and
therefore $P$ belong to $<x_1^2 +x_1,...., x_n^2 +x_n>$.
\end{proof}

Therefore the ring $\mathbb{Z}_2[\mathbb{A}^n]$ of regular functions on $\mathbb{Z}_2^n$ can be identified with the quotient ring
$$\mathbb{Z}_2[\mathbb{A}^n] \ =  \ \mathbb{Z}_2[x_1,...,x_n] /  <x_1^2 +x_1,...., x_n^2 +x_n> .$$

Often we think of $\mathbb{Z}_2^n$ as a ring, with coordinate-wise sum and product.  We identify $\mathbb{Z}_2^n$ with $\mathrm{P}[n]$, the set of subsets of $[n]$, as follows: $a \in \mathbb{Z}_2^n$ is identified with the  subset $ a \subseteq [n]$
such that $i\in a$ iff $a_i=1$. With this identification the product $ab$ of elements in $\mathbb{Z}_2^n$
agrees with the intersection $a \cap b$ of the sets $a$ and $b$; the sum $a+b$ corresponds with the symmetric difference
$a+b = (a \cup b) \setminus (a\cap b)$; the element $a +(1,...,1)$ is identified with the complement $\overline{a}$ of $a$.
Note that $a \cup b = a+b+ab$. We let $\mathrm{P} \mathrm{P}[n]$ be the set of families of subsets of $[n]$.\\

For $a\in \mathrm{P} [n], \ $ let $m^a:\mathbb{Z}_2^n \longrightarrow \mathbb{Z}_2$ be the map such that
$$m^a(b)\ = \
\left\{\begin{array}{cc} 1 & \mathrm{if} \ a = b,\\
0 & \ \mathrm{otherwise.}  \end{array}\right.$$
For $a  \in \mathrm{P}[n]$ non-empty, let $x^a \in \mathbb{Z}_2[x_1,...,x_n]$ be the monomial $x^a= \prod_{i \in a} x_i$. Also
set $x^{\emptyset}=1.$ The monomial  $x^a$
defines the map $x^a:\mathbb{Z}_2^n \longrightarrow \mathbb{Z}_2$ given by $$x^a(b)\ = \
\left\{\begin{array}{cc} 1 & \mathrm{if} \ a \subseteq b,\\
0 & \ \mathrm{otherwise.}  \end{array}\right.$$
For $a  \in \mathrm{P}[n]$ non-empty, let $w^a \in \mathbb{Z}_2[x_1,...,x_n]$ be the monomial $w^a= \prod_{i \in a} (x_i + 1)$.
Also set $w^{\emptyset}=1.$ The monomial $w^a$
defines the map $w^a:\mathbb{Z}_2^n \longrightarrow \mathbb{Z}_2$ given by $$w^a(b)\ = \
\left\{\begin{array}{cc} 1 & \ \mathrm{if} \  b \subseteq \overline{a} ,\\
0 & \ \ \mathrm{otherwise.} \end{array}\right.$$

Lemma \ref{bases}  below follows from the definitions above and the M$\ddot{\mbox{o}}$bius inversion formula \cite{GCRota}, which can be stated as follows. Given maps $f,g: \mathrm{P}[n] \longrightarrow R$, with $R$ a ring of characteristic $2$, then
$$f(b)\ = \ \sum_{a \subseteq b}g(a) \ \ \ \ \ \mbox{if and only if} \ \ \ \ \  g(b)\ = \ \sum_{a \subseteq b}f(a).$$

\begin{lem}\label{bases}{\em  The following identities hold in $\mathbb{Z}_2[\mathbb{A}^n]$:
$$1) \ m^a=x^aw^{\overline{a}}.\ \ \ \ \ \ 2) \ x^a = \sum_{a \subseteq b} m^b. \ \ \ \ \ \ 3) \ m^a= \sum_{a \subseteq b}x^b. \ \ \ \ \ \  4) \ w^b = \sum_{a \subseteq \overline{b}}m^a.$$
$$ 5) \ m^a = \sum_{\overline{a} \subseteq b}w^b. \ \ \ \ \ \ 6) \  w^b = \sum_{a \subseteq b}x^a . \ \ \ \ \ \ 7) \  x^b = \sum_{a \subseteq b}w^a.$$
$$8) \ m^am^b = \delta_{ab}m^a .\ \ \ \ \ \ 9) \ x^ax^b = x^{a \cup b} .  \ \ \ \ \ \ 10)  \  w^a w^b = w^{a\cup b} .$$}
\end{lem}

Note that $ \mathbb{Z}_2[\mathbb{A}^n] = \mathrm{M}(\mathbb{Z}_2^n, \mathbb{Z}_2)$, indeed a map $f:\mathbb{Z}_2^n \longrightarrow \mathbb{Z}_2 $ can be written as
$$f\ = \ \sum_{ \ f(a)=1}m^a\ = \ \sum_{ \ f(a)=1} x^a w^{\overline{a}} \ = \ \sum_{ \ f(a)=1}\prod_{i \in a}x_i\prod_{i \in \overline{a}}(x_i +1)$$
$$= \ \sum_{f(a)=1, \ b \subseteq \overline{a}}  x^{a \cup b}\ = \ \sum_{f(a)=1, \ a \subseteq b}  x^{ b}.$$

From Lemma \ref{bases} we see that there are several natural bases for the $\mathbb{Z}_2$-vector space
$$\mathbb{Z}_2[\mathbb{A}^n] \ \ = \ \ \mathbb{Z}_2[x_1,...,x_n] /  <x_1^2 +x_1,...., x_n^2 +x_n> ,$$
namely we can pick $\ \ \{m^a \ | \ a \in \mathrm{P} [n]  \}, \ \   \{x^a \ | \ a \in \mathrm{P} [n]  \},\ \ \mbox{or} \ \ \{w^a \ | \ a \in \mathrm{P} [n]  \}.$ We
use the following notation to write the coordinates of $f \in \mathbb{Z}_2[\mathbb{A}^n]$ in each one of these bases
$$f \ = \ \sum_{a \in \mathrm{P}[n]}f(a)m^a \ = \ \sum_{a \in \mathrm{P} [n]}f_x(a)x^a \ = \ \sum_{a \in \mathrm{P}[n]}f_w(a)w^a.$$
We obtain three linear maps $f \longrightarrow f, \ f \longrightarrow f_x, \ \mbox{and} \ f \longrightarrow f_w$ from $\mathbb{Z}_2[\mathbb{A}^n]$
to $\mathrm{M}(\mathbb{Z}_2^n, \mathbb{Z}_2)$.  The coordinates $f$, $f_x$ and $f_w$ are connected, via the M$\ddot{\mbox{o}}$bius inversion formula,  by the relations: $$f_x(b)\ = \ \sum_{a \subseteq b} f(a), \ \ \ \ \ \ f(b)\ = \ \sum_{a \subseteq b} f_x(a),
\ \ \ \ \ \ f_w(b)\ = \ \sum_{a \subseteq b} f(\overline{a}),$$
$$f(b)\ = \ \sum_{a \subseteq \overline{b}} f_w(a), \ \ \ \ \ \ f_x(a)\ = \ \sum_{a \subseteq b} f_w(b) , \ \ \ \ \ \ f_w(a)\ = \ \sum_{a \subseteq b} f_x(b).$$

The maps $f \longrightarrow f_x \ \ \mbox{and} \ \ f \longrightarrow f_w$ fail to be ring morphisms. Instead we have the identities:
$$(fg)_x(c)\ = \ \sum_{a \cup b =c}f_x(a)g_x(b) \ \ \ \ \ \ \ \ \mbox{and} \ \ \ \ \ \ \ \ (fg)_w(c)\ = \ \sum_{a \cup b =c}f_w(a)g_w(b).$$

We define a predicate $O$ on finite sets as follows: given a finite set $a$,  then
$Oa$ holds if and only if the cardinality of $a$ is an odd number. In other words, $O$ is the map from
finite sets to $\mathbb{Z}_2$ such that $Oa=1$ if and only if the cardinality of $a$ is odd.

\begin{exmp}{\em Let $C \in \mathrm{P}\mathrm{P}[n]$. An ordered $k$-covering of $a\in \mathrm{P}[n]$ by elements of $C$   is a tuple $c_1,...,c_k \in C$ such that
$$c_1 \cup \cdots \cup c_k \ = \ a.$$ Let $k$-$Cov_C(a)$ be the set of $k$-coverings of $a$  by elements of $C$.
Then $a\in \mathrm{P}[n]$ belongs to $C$ if and only if  $|k$-$Cov_C(a)|$ is odd for every $k \geq 1$. Indeed, let $f \in \mathbb{Z}_2[\mathbb{A}^n]$
be given by $$f\ = \ \sum_{c \in C}x^c\ = \sum_{a \in \mathrm{P} [n]}1_C(a)x^a,$$ where
$1_C:  \mathrm{P} [n] \longrightarrow \mathbb{Z}_2$ is the characteristic function of $C$. Since $f^k=f$ for every
$f \in \mathbb{Z}_2[\mathbb{A}^n]$, we have that $$\sum_{a \in \mathrm{P} [n]}1_C(a)x^a \ = \ f \ = \ f^k \ = \
\sum_{a \in \mathrm{P} [n]}\left( \sum_{a_1 \cup \cdots \cup a_k = a}\prod_{i=1}^k 1_C(a_i)\right)x^a \ = \
\sum_{a \in \mathrm{P} [n]} O(k\mbox{-}Cov_C(a)) x^a .$$
We conclude that $1_C(a)=O(k\mbox{-}Cov_C(a))$, and thus $a \in C$ if and only if  $|k\mbox{-}Cov_C(a)|$ is odd.
}
\end{exmp}

\section{Differential Operators on Boolean Affine Spaces}\label{s2}

Next we consider the algebra of differential operators on affine Boolean spaces. Note that the partial derivatives $\partial_i$ on $\mathbb{Z}_2[x_1,...,x_n]$ do not descent to well-defined operators on $\mathbb{Z}_2[\mathbb{A}^n]$; indeed if we had such an
operator, then $0= x_i + x_i = \partial_i x_i^2 = \partial_i x_i =1.$ It is thus necessary to introduce an alternative definition for the partial derivatives $\partial_i$ on $\mathbb{Z}_2[\mathbb{A}^n]$.\\

The Boolean partial derivative $\partial_i f:\mathbb{Z}_2^n \longrightarrow \mathbb{Z}_2$ of a map
$f:\mathbb{Z}_2^n \longrightarrow \mathbb{Z}_2$ is given \cite{bro, reed} by
$$ \partial_i f(x) \ = \ f(x + e_i) + f(x)$$ where $e_i \in \mathbb{Z}_2^n$ is the vector with vanishing entries except at position $i$.
With this definition  $\partial_i x_i = \partial_i x_i^2=1$, the contradiction above does not arise, and we obtained
well-defined operators  $$\partial_i : \mathbb{Z}_2[\mathbb{A}^n] \ \longrightarrow \ \mathbb{Z}_2[\mathbb{A}^n].$$
The operators $\partial_i$ fail to be derivations; instead they  satisfy the twisted Leibnitz
identity  $$\partial_i(fg)\ = \ (\partial_if) g + (s_if)(\partial_ig)$$ where the shift operators $s_i : \mathbb{Z}_2[\mathbb{A}^n] \longrightarrow \mathbb{Z}_2[\mathbb{A}^n]$ are given
by $s_if(x) = f(x + e_i).$ Indeed:
\begin{eqnarray*}
\partial_i(fg)(x)&=& f(x+e_i)g(x+ e_i) \ + \ f(x)g(x)\\
&=& [f(x+ e_i) + f(x)]g(x) \ + \ f(x+e_i)[ g(x+ e_i)+g(x)]\\
&=&\partial_if(x) g(x) \ + \ s_if(x)\partial_ig(x).
\end{eqnarray*}
The operators $\partial_i$ are nilpotent, i.e. $\partial_i^2=0,$ since:
$$\partial_i^2f(x)\ = \ \partial_if(x +e_i) + \partial_if(x)\ = \ f(x) + f(x +e_i) + f(x + e_i) + f(x)\ = \ 0.$$

\begin{defn}
{\em The $\mathbb{Z}_2$-algebra $\mathrm{BDO}_n$ of Boolean differential operators on $\mathbb{Z}_2^n$ is the $\mathbb{Z}_2$-subalgebra of
$\mathrm{End}_{\mathbb{Z}_2}(\mathbb{Z}_2[\mathbb{A}^n])$ generated by $\partial_i$ and the operators of multiplication by $x_i$ for $i \in [n].$}
\end{defn}

\begin{thm}\label{rel}{\em  The following identities hold for $x_i, \partial_i, s_i \in \mathrm{BDO}_n$ and $i\in [n]$:
$$1.\ x_i^2=x_i. \ \ \ \ \ 2. \ \partial_i^2=0. \ \ \ \ \ 3. \ s_i^2=1. \ \ \ \ \ 4. \ \partial_i = s_i +1. \ \ \ \ \ 5. \  \partial_is_i = s_i\partial_i = \partial_i.$$
$$6. \ s_i = \partial_i +1. \ \ \ 7. \ s_ix_i = x_is_i +s_i=(x_i+1)s_i. \ \ \ 8.  \  \partial_ix_i=x_i\partial_i +s_i = x_i\partial_i + \partial_i +1.$$}
\end{thm}
\begin{proof}
We have already shown that $x_i^2=x_i$ and $\partial_i^2=0$.  For the other identities we have that:\\

\noindent $\bullet$ $s_i^2f(x)= s_if(x + e_i)=f(x + e_i +e_i)=f(x);$\\

\noindent $\bullet$ $\partial_if(x)=f(x+ e_i)+ f(x)=s_if(x)+f(x)=(s_i +1)f(x);$\\

\noindent $\bullet$ $s_i\partial_if(x)=\partial_if(x +e_i)=f(x+e_i + e_i)+f(x+e_i)=f(x)+f(x+ e_i)=\partial_if(x) ;$\\

\noindent $\bullet$ $\partial_is_if(x)=s_if(x+e_i)+s_if(x)=f(x+e_i+e_i)+ f(x+e_i)=f(x)+f(x+ e_i)=\partial_if(x) ;$\\

\noindent $\bullet$ $s_ix_if(x)=(x_i+1)f(x+e_i)= x_if(x + e_i) + f(x + e_i)=x_is_if(x) + s_if(x)=(x_is_i + s_i)f(x);$\\

\noindent $\bullet$ $s_if(x)=f(x +e_i)= f(x +e_i) +f(x)+f(x)=\partial_if(x) + f(x)= (\partial_i +1)f(x);$\\

\noindent $\bullet$ $\partial_i(x_if)(x)=x_if(x+ e_i) + f(x + e_i) + x_if(x) =x_i(f(x+ e_i) + f(x))+ f(x + e_i),$ thus \\

\noindent $\bullet$ $\partial_i(x_if)=x_i\partial_if + f(x +e_i)=(x_i\partial_i + s_i)f=(x_i\partial_i + \partial_i +1)f.$

\end{proof}

The operator $\partial_i$ acts on the bases $m^a, \ x^a \ \mbox{and} \ w^a $
as follows:
$$\partial_i m^a = m^{a + e_i}+ m^a, \ \ \ \ \partial_ix^a=
\left\{\begin{array}{cc} x^{a \setminus i} & \mathrm{if}\  i\in a \\
0 & \ \ \ \mathrm{otherwise,}  \end{array}\right. \ \ \mbox{and} \ \ \ \ \partial_iw^a=
\left\{\begin{array}{cc}  w^{a \setminus i} & \mathrm{if}\  i\in a \\
0 & \ \ \ \ \mathrm{otherwise}.  \end{array}\right.$$

\noindent From these expressions we obtain that:\\

\noindent $\bullet \ \partial_i f(a)=1$ if and only if $f(a)\neq f(a+e_i)$, that is
$$\partial_i f \ \ = \ \ \underset{{a \in \mathrm{P} [n], \ f(a)\neq f(a + e_i)}} {\sum}m^a .$$

\noindent $\bullet  \  (\partial_i f)_x(a)=f_x(a \cup i)$ if $i \notin a$ and $ (\partial_i f)_x(a)=0$ if $i \in a$, that is
$$\partial_i f\ \ =\ \ \underset{{i \in a \in \mathrm{P} [n]}} {\sum}f_x(a)x^{a-i} .$$
\noindent $\bullet   \  (\partial_i f)_w(a)=f_w(a \cup i)$ if $i \notin a$ and $ (\partial_i f)_w(a)=0$ if $i \in a$, that is
$$\partial_i f\ \ = \ \  \underset{{i \in a \in \mathrm{P} [n]}} {\sum}f_w(a)w^{a-i}.$$

More generally one can show by induction, for $a,b \in \mathrm{P}[n],$ that:
$$\partial^b m^a = \sum_{c \subseteq b}m^{a + c}, \ \ \ \partial^bx^a=
\left\{\begin{array}{cc} x^{a \setminus b} & \mathrm{if}\  b \subseteq a \\
0 & \ \ \ \mathrm{otherwise,}  \end{array}\right. \ \ \ \mbox{and}\ \ \ \  \partial^bw^a=
\left\{\begin{array}{cc} w^{a \setminus b} & \mathrm{if}\  b \subseteq a \\
0 & \ \ \ \mathrm{otherwise.}  \end{array}\right.$$

By definition  $\mathrm{BDO}_n \subseteq \mathrm{End}_{\mathbb{Z}_2}(\mathbb{Z}_2[\mathbb{A}^n])$ acts naturally on
$\mathbb{Z}_2[\mathbb{A}^n]$, so we get a map
$$\mathrm{BDO}_n \otimes_{\mathbb{Z}_2}\mathbb{Z}_2[\mathbb{A}^n] \longrightarrow \mathbb{Z}_2[\mathbb{A}^n].$$

\begin{prop}\label{do}
{\em Consider maps $D:\mathrm{P}[n]\times \mathrm{P}[n] \longrightarrow \mathbb{Z}_2\ $ and $\ f:\mathrm{P}[n] \longrightarrow \mathbb{Z}_2$.\\

\noindent 1. Let $D=\underset{{a,b \in \mathrm{P}[n]}} {\sum}D(a,b)m^a\partial^b \in \mathrm{BDO}_n$,
$f=\underset{{c \in \mathrm{P} [n]}} {\sum}f(c)m^c \in \mathbb{Z}_2[\mathbb{A}^n]$, and
$Df=\underset{{a \in \mathrm{P}[n]}} {\sum}Df(a)m^a $. Then we have that $$Df(a)\  = \ {\underset{ e \subseteq b }{\sum}}D(a,b)f(a+e). $$
\noindent 2. Let $D=\underset{{a,b \in \mathrm{P} [n]}} {\sum}D_x(a,b)x^a\partial^b \in \mathrm{BDO}_n$, \
$f=\underset{{c \in \mathrm{P} [n]}} {\sum}f_x(c)x^c \in \mathbb{Z}_2[\mathbb{A}^n]$, \ and
$Df=\underset{{e \in \mathrm{P} [n]}} {\sum}Df_x(e)x^e$. Then
$$Df_x(e)\ = \ \underset{{\underset {a \cup (c\setminus b)=e}{a,\ b \subseteq c}}} {\sum} D_x(a,b)f_x(c) .$$
}
\end{prop}

\begin{proof}
$$1. \ \ \ \ \ \ Df\ = \ \underset{{a,b,c  \in \mathrm{P}[n]}} {\sum}D(a,b)f(c)m^a\partial^bm^c\ = \ \underset{{a,e \subseteq b,c  }} {\sum}D(a,b)f(c)m^am^{c+e}$$ $$= \ \underset{{a,e \subseteq b }} {\sum}D(a,b)f(a+e)m^a\ = \
\sum_{a  \in \mathrm{P}[n] } \left( \sum_{e \subseteq b }D(a,b)f(a+e) \right)m^a.$$

$$2. \ \ \ \ \ \ Df\ = \ \underset{{a,b,c  \in \mathrm{P}[n]}} {\sum}D_x(a,b)f_x(c)x^a\partial^bx^c \ = \ \underset{{a, b \subseteq c  }} {\sum}D_x(a,b)f_x(c)x^{a \cup c \setminus b}$$
$$= \sum_{{e  \in \mathrm{P}[n]}} \left(  \underset{{\underset {a \cup (c\setminus b)=e}{a,\ b \subseteq c}}} {\sum} D_x(a,b)f_x(c)  \right)x^e.$$

\end{proof}

\begin{thm}\label{egre}
{\em For $n \geq 1$ we have that $\ \mathrm{BDO}_n \ = \ \mathrm{End}_{\mathbb{Z}_2}(\mathbb{Z}_2[\mathbb{A}^n]).$
}
\end{thm}

\begin{proof}
Note that $$\mathrm{dim}\left( \mathrm{End}_{\mathbb{Z}_2}(\mathbb{Z}_2[\mathbb{A}^n]) \right) \ = \
\mathrm{dim}(\mathbb{Z}_2[\mathbb{A}^n])\mathrm{dim}(\mathbb{Z}_2[\mathbb{A}^n])\ = \ 2^n2^n\ = \ 2^{2n}.$$ The set
$\{  x^a\partial^b \ | \ a,b \in \mathrm{P}[n] \}$ has $2^{2n}$ elements and generates $\mathrm{BDO}_n$
as a vector space over $\mathbb{Z}_2$; thus it is enough to show that it is a linearly independent set. Suppose that
$$\underset{{a,b \in \mathrm{P} [n]}} {\sum}f(a,b)x^a\partial^b\ = \ \sum_{b \in \mathrm{P} [n]}\left( \sum_{a \in \mathrm{P} [n]}f(a,b)x^a \right) \partial^b \ = \ 0.$$ Pick a minimal set $c \in \mathrm{P}[n]$  such that $\underset{{a \in \mathrm{P} [n]}} {\sum}f(a,c)x^a  \neq 0.$ We have that:
$$\left( \sum_{a, b \in \mathrm{P}[n]}f(a,b)x^a  \partial^b \right)(x^c)\ = \
\sum_{b \in \mathrm{P} [n]} \left( \sum_{a \in \mathrm{P} [n]}f(a,b)x^a \right) \partial^b (x^c) \ = \
\sum_{a \in \mathrm{P} [n]}f(a,c)x^a \ = \ 0.$$
Therefore, since $\{x^a \ | \ a \in \mathrm{P} [n] \}$ is a basis for $\mathbb{Z}_2[\mathbb{A}^n]$, we have that $f(a,c)=0$ in contradiction
with the fact $\underset{{a \in \mathrm{P} [n]}} {\sum}f(a,c)x^a  \neq 0.$ We conclude that $\mathrm{dim}\left(\mathrm{BDO}_n \right)=2^{2n}$
yielding the desired result.

\end{proof}

Putting together Proposition \ref{do} and Theorem \ref{egre} we get a couple of explicit ways of identifying $\mathrm{BDO}_n$ with
$\mathrm{M}_{2^n}(\mathbb{Z}_2)$, the algebra of square matrices of size $2^n$ with coefficients in $\mathbb{Z}_2.$ Note
that  $\mathrm{M}_{2^n}(\mathbb{Z}_2)$ may be identified with $\mathrm{M}(\mathrm{P}[n]\times \mathrm{P}[n], \mathbb{Z}_2).$
Moreover, we can identify $\mathrm{M}(\mathrm{P}[n]\times \mathrm{P}[n], \mathbb{Z}_2)$ with the set of directed graphs
with vertex set $\mathrm{P}[n]$ and without multiple edges as follows: given a matrix $M \in \mathrm{M}_{2^n}(\mathbb{Z}_2)$ its associated graph  has
an edge from $b$ to $a$ if and only if $M_{a,b}=1.$\\

Let $\mathrm{R}:\mathrm{BDO}_n \longrightarrow  \mathrm{M}_{2^n}(\mathbb{Z}_2)$ be the $\mathbb{Z}_2$-linear map constructed as follows.
Consider the bases $$\{  m^a\partial^b \ | \ a,b \in \mathrm{P} [n]  \} \ \  \mbox{for} \ \  \mathrm{BDO}_n \ \ \ \ \ \mbox{and} \ \ \ \ \
\{  m^a \ | \ a \in \mathrm{P}[n]  \} \ \  \mbox{for} \ \ \mathbb{Z}_2[\mathbb{A}^n].$$
For $a,b \in \mathrm{P}[n]$, let $\mathrm{R}(m^a\partial^b)$ be the matrix of $m^a\partial^b$ on the basis $m^a$.
The action of $m^a\partial^b$ on $m^c$ is given by $$m^a\partial^bm^c\ = \ m^a\sum_{e \subseteq b}m^{c + e}\ = \ \sum_{e \subseteq b}m^am^{c + e}\ = \ m^a \ \ \mbox{if} \ \ c+a \subseteq b \ \mbox{ and zero otherwise}.$$ Therefore, the matrix $\mathrm{R}(m^a\partial^b)$ is given
for $c,d \in \mathrm{P}[n]$ by the rule
    $$\mathrm{R}(m^a\partial^b)_{c,d}\ = \
\left\{\begin{array}{cc} 1 & \mathrm{if}\  c=a \ \mbox{and }\ d+a \subseteq b, \\
0 & \ \mathrm{otherwise.}  \end{array}\right. $$

\begin{exmp}{\em  The graph of $\mathrm{R}(m^{\{1,2\}}\partial^{\{2,3 \}})$ is show in Figure \ref{g1}.
}
\end{exmp}

\begin{figure}[ht]
\begin{center}
\includegraphics[height=3cm]{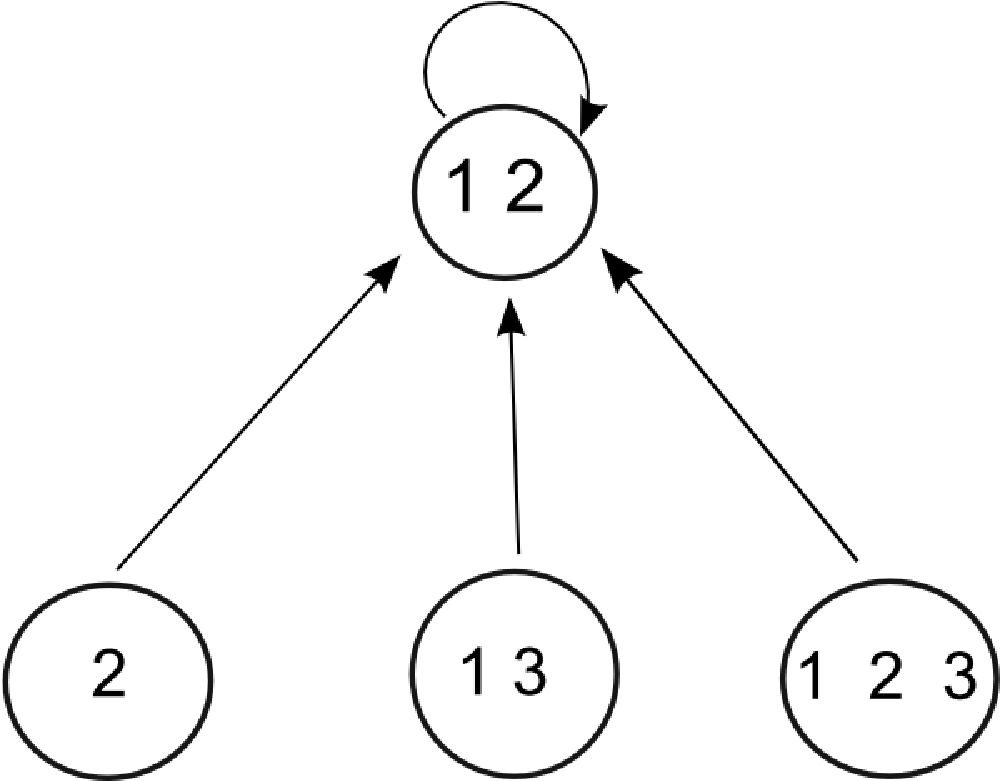}
\caption{ \ Graph of the matrix $\mathrm{R}(m^{\{1,2\}}\partial^{\{2,3 \}})$. \label{g1}}
\end{center}
\end{figure}

For a second representation consider the $\mathbb{Z}_2$-linear
map $\mathrm{S}:\mathrm{BDO}_n \longrightarrow  \mathrm{M}_{2^n}(\mathbb{Z}_2)$ constructed as follows.
Consider the bases $$\{  x^a\partial^b \ | \ a,b \in \mathrm{P} [n] \} \ \ \mbox{for} \ \  \mathrm{BDO}_n  \ \ \ \ \ \mbox{and} \ \ \ \ \
\{  x^a \ | \ a \in \mathrm{P}[n]  \} \ \ \mbox{for} \ \ \mathbb{Z}_2[\mathbb{A}^n].$$
For $a,b \in \mathrm{P}[n]$ let $\mathrm{S}(x^a\partial^b)$ be the matrix of $x^a\partial^b$ on the basis $x^a$.
The action of $x^a\partial^b$ on $x^c$ is given by
$$x^a\partial^bx^c\  = \ \left\{\begin{array}{cc} x^{a \cup c\setminus b} & \mathrm{if}\ b \subseteq c, \\
0 & \ \mathrm{otherwise.}  \end{array}\right.$$
Therefore, the matrix $\mathrm{S}(x^a\partial^b)$ is given for $c,d \in \mathrm{P}[n]$ by the rule
    $$\mathrm{S}(x^a\partial^b)_{c,d}\ = \
\left\{\begin{array}{cc} 1 & \mathrm{if}\  c=a \cup d\setminus b \ \mbox{and }\ b \subseteq d, \\
0 & \ \mathrm{otherwise.}  \end{array}\right. $$

\begin{exmp}{\em The graph associated to the matrix $\mathrm{S}(m^{\{1\}}\partial^{\{3 \}})$ is shown in Figure 2.}
\end{exmp}

\begin{figure}[ht]
\begin{center}
\includegraphics[height=3.3cm]{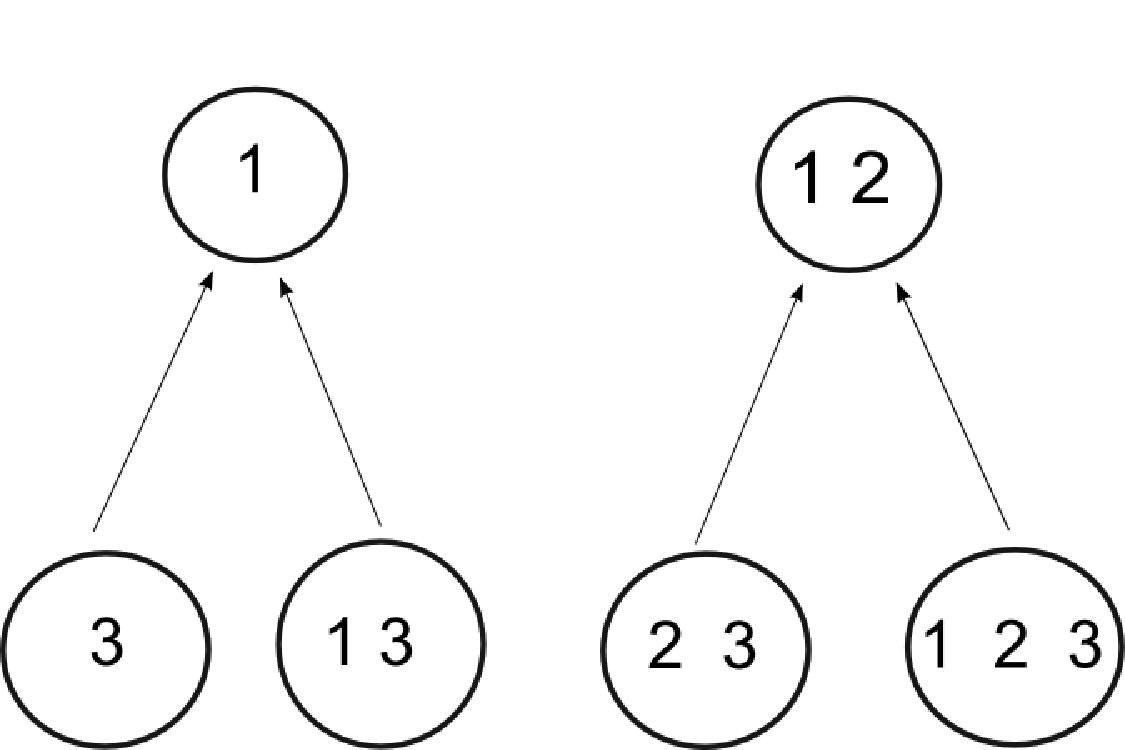}
\caption{ \ Graph of the matrix $\mathrm{S}(m^{\{1\}}\partial^{\{3 \}})$. \label{bra1}}
\end{center}
\end{figure}

\section{ Boole-Weyl Algebras}\label{bwa}

First we motivate, from the viewpoint of canonical quantization, our definition of Boole-Weyl algebras.
Canonical phase space, for a field $k$ of characteristic zero, can be identified with the affine space $k^n\times k^n$.
The Poisson bracket on $k[x_1,...,x_n, y_1,...,y_n]$ in canonical coordinates $x_1,...,x_n, y_1,...,y_n\ $ on $\ k^n\times k^n$
is given by $$\{ x_i , x_j \}=0,  \ \ \ \ \ \{ y_i , y_j \}=0,\ \ \ \ \   \{ x_i , y_j \}= \delta_{i,j}.$$
Equivalently, the Poisson bracket is given for  $f , g \in k[x_1,...,x_n, y_1,...,y_n]$  by
$$\{ f , g \}\ = \ \sum_{i=1}^n\frac{\partial f}{\partial x_i}\frac{\partial g}{\partial y_i} -
\frac{\partial f}{\partial y_i}\frac{\partial f}{\partial x_i}.$$
Canonical quantization may be formulated as the problem of promoting the commutative variables $x_i$ and $y_j$
into non-commutative operators $\ \widehat{x}_i \ $ and  $\ \widehat{y}_j\ $ satisfying the commutation relations:
$$ [\widehat{x}_i , \widehat{x}_j ]= 0, \ \ \ \ \  \ [ \widehat{y}_i  , \widehat{y}_j ]=0, \ \ \ \ \ \
[ \widehat{y}_i  , \widehat{x}_j ]= \delta_{i,j}.$$
Note that the free algebra generated  by
$\ \widehat{x}_i \ $ and  $\ \widehat{y}_j\ $ subject to the above relations is precisely what is called the Weyl algebra.\\

Now let $k=\mathbb{Z}_2$ and consider the affine phase spaces $\mathbb{Z}_2^n\times \mathbb{Z}_2^n$. Let $x_1,...,x_n, y_1,...,y_n$
be canonical coordinates on $\mathbb{Z}_2^n\times \mathbb{Z}_2^n$. The analogue of the Poisson bracket
$$\{ \ , \ \} : \mathbb{Z}_2[\mathbb{A}^{2n}] \otimes  \mathbb{Z}_2[\mathbb{A}^{2n}]\ \longrightarrow \ \mathbb{Z}_2[\mathbb{A}^{2n}]$$
can be expressed for $f,g \in \mathbb{Z}_2[\mathbb{A}^{2n}]$ as
$$\{ f , g \}\ = \ \sum_{i=1}^n\frac{\partial f}{\partial x_i}\frac{\partial g}{\partial y_i} +
\frac{\partial f}{\partial y_i}\frac{\partial f}{\partial x_i},$$
where $\frac{\partial }{\partial x_i}$ and $\frac{\partial }{\partial y_i}$ are the Boolean partial derivatives
along the coordinates $ x_i$ and $y_i.$
Clearly, the full set of axioms for a Poisson bracket will not longer hold, e.g.
Boolean derivatives are twisted derivations. Nevertheless, the bracket is still determined by its values on the
canonical coordinates: $\{ x_i , x_j \}=0, \ \ \{ y_i , y_j \}=0, \ \ \{ x_i , y_j \}= \delta_{i,j}.$
Canonical quantization consists in promoting the commutative variables $x_i$ and $y_j$
to non-commutative operators $\widehat{x}_i$ and $\widehat{y}_j$ satisfying the commutation relations:
$$ [\widehat{x}_i , \widehat{x}_j ]= 0, \ \ \ \ [ \widehat{y}_i , \widehat{y}_j  ]=0, \ \ \ \ [ \widehat{y}_i , \widehat{x}_j ]=0
\ \ \mbox{for} \ i\neq j, \ \ \ \mbox{and} \ \ \ \ \ \
[ \widehat{y}_i , \widehat{x}_i  ]_{s_i}= 1.$$
Note that in the last relation we did not use the commutator but the twisted commutator
$$[f,g]_{s_i}\ =\ fg \ + \ (s_if)g ;$$ this choice is expected since the operators $\widehat{y}_i$ are twisted
derivations instead of usual derivations.  The relation
$[ \widehat{y}_i , \widehat{x}_i  ]_{s_i}= 1$ can be equivalently written using commutators as
$$ [ \widehat{y}_i , \widehat{x}_i  ] \ = \ \widehat{y}_i +1.$$

We are ready to introduce the Boole-Weyl algebras $\mathrm{BW}_n$, which we also call quantum Boolean algebras.
The Boole-Weyl algebra $\mathrm{BW}_n$ is the  free algebra generated  by
$x_i$ and $y_j$ subject to the relations  above (removing the unnecessary hats). The algebras  $\mathrm{BW}_n$ are the analogue of the Weyl algebras in the Boolean context.

\begin{defn}
{\em The algebra $\mathrm{BW}_n$ is the quotient of $\mathbb{Z}_2<x_1,...,x_n, y_1,...,y_n>$,
the free associative $\mathbb{Z}_2$-algebra generated by $x_1,...,x_n, y_1,...,y_n$, by the ideal
$$<x_i^2 +x_i, \ x_ix_j + x_jx_i, \ y_iy_j + y_jy_i, \ y_i^2, \ y_ix_j + x_jy_i, \ y_ix_i + x_iy_i + y_i +1>,  $$
generated by the relations $x_i^2 =x_i, \ y_i^2=0, \ $ and $\ y_ix_i = x_iy_i + y_i+1 \ $ for $i \in [n], \ $
$x_ix_j = x_jx_i \ $ and $\ y_iy_j = y_jy_i$ for $i,j \in [n], \ $ and $\ y_ix_j + x_jy_i$ for $i \neq j \in [n].$
}
\end{defn}

\begin{thm}\label{qbas}{\em The map $ \mathbb{Z}_2<x_1,...,x_n, y_1,...,y_n> \ \longrightarrow \ \mathrm{End}_{\mathbb{Z}_2}(\mathbb{Z}_2[\mathbb{A}^n])$ sending $x_i$ to the operator of multiplication by $x_i$, and $y_i$ to $\partial_i$, descends to an
 isomorphism  $ \mathrm{BW}_n \longrightarrow \mathrm{End}_{\mathbb{Z}_2}(\mathbb{Z}_2[\mathbb{A}^n])$ of $\mathbb{Z}_2$-algebras.}
\end{thm}

\begin{proof}By Theorem \ref{rel} the given map descends. By definition it is a surjective
map $$ \mathrm{BW}_n \ \longrightarrow \ \mathrm{BDO}_n= \mathrm{End}_{\mathbb{Z}_2}(\mathbb{Z}_2[\mathbb{A}^n]).$$
Moreover, this map is an isomorphisms since $ \mathrm{dim}\left( \mathrm{BW}_n \right)=
\mathrm{dim}\left( \mathrm{End}_{\mathbb{Z}_2}(\mathbb{Z}_2[\mathbb{A}^n])\right)$. Indeed
using the commutation relations it is easy to check that the natural map
$$\mathbb{Z}_2[x_1,...,x_n]/<x_i^2 +x_i>\otimes \ \mathbb{Z}_2[y_1,...,y_n]/<y_i^2> \ \ \ \longrightarrow \ \ \ \mathrm{BW}_n$$
is surjective. If $\underset{{a, b \in \mathrm{P} [n]}} {\sum}f(a,b)x^a \otimes y^b$ is in the kernel of the latter
map, then the Boolean differential operator
$\underset{{a, b \in \mathrm{P} [n]}} {\sum}f(a,b)x^a \partial^b$
would vanish, and therefore the coefficients $f(a,b)$ must vanish as well.
Thus $$\mathrm{dim}(\mathrm{BW}_n) \ = \ \mathrm{dim}(\mathbb{Z}_2[x_1,...,x_n]/<x_i^2 +x_i>)\mathrm{dim}(\mathbb{Z}_2[y_1,...,y_n]/<y_i^2>)\ = \ 2^n2^n$$
$$=\ \mathrm{dim}(\mathbb{Z}_2[\mathbb{A}^n])\mathrm{dim}(\mathbb{Z}_2[\mathbb{A}^n])\ = \
\mathrm{dim}(\mathrm{End}_{\mathbb{Z}_2}(\mathbb{Z}_2[\mathbb{A}^n])).$$

\end{proof}

\begin{thm}{\em The map $ \mathbb{Z}_2<x_1,...,x_n, y_1,...,y_n> \  \longrightarrow \   \mathrm{End}_{\mathbb{Z}_2}(\mathbb{Z}_2[\mathbb{A}^n])$ sending $x_i$ to the operator of multiplication by $w_i=x_i +1$, and $y_i$ to the operator $\partial_i$, descends to an isomorphism  $ \mathrm{BW}_n \longrightarrow \mathrm{End}_{\mathbb{Z}_2}(\mathbb{Z}_2[\mathbb{A}^n])\ $ of $\ \mathbb{Z}_2$-algebras.}
\end{thm}

\begin{proof}
Follows from the fact that $w_i$ and $\partial_j$ satisfy exactly the same relation as $x_i$ and $\partial_j.$
\end{proof}

\begin{cor}\label{coro}{\em Any identity in $\mathrm{BW}_n$  involving $x_i$ and $\partial_j$ has an associated
identity involving $w_i$ and $\partial_j$ obtained by replacing $x_i$ by $w_i.$}
\end{cor}

\begin{lem}{\em For $a,b, c, d \in \mathrm{P}[n]$ the following identities hold in $\mathrm{BW}_n$ :
$$1. \ y^bm^c\ = \ \underset{{b_1 \subseteq b_2 \subseteq b}} {\sum}m^{c+b_2}y^{b_1}. \ \ \ \ \ \ \
2. \ m^ay^bm^cy^d\ = \ \underset{{\underset {e\setminus d \subseteq a+c \subseteq b }{d\subseteq e}}} {\sum}m^ay^{e}.
\ \ \ \ \ \ \ \ \ \ \ \ \ \ \ $$
$$3. \ y^bx^c\ =\ \underset{{k_1 \subseteq k_2 \subseteq b \cap c}} {\sum}x^{c\setminus k_2}y^{b\setminus k_1}. \ \ \ \ \ \
4. \ x^ay^bx^cy^d \ = \ \underset{{a \subseteq e,\ d \subseteq f}} {\sum}c(a,b,c,d,e,f)x^ey^f,$$
where
$$ c(a,b,c,d,e,f)\ = \
O\Big\{ k_1 \subseteq k_2 \subseteq b \cap c \ \  | \ \ a \cup ( c \setminus k_2) = e ,\  \  b \setminus k_1= f \setminus d \Big\}.$$
}

\end{lem}

\begin{proof}1. By Theorem \ref{qbas} it is enough to show that the differential operators associated with
 both sides of the equation are equal. Consider the operator of multiplication by $f:\mathbb{Z}_2^n \longrightarrow \mathbb{Z}_2$ and let
$g:\mathbb{Z}_2^n \longrightarrow \mathbb{Z}_2$ be any other map. The twisted Leibnitz rule $$\partial_i(fg)\ = \ (\partial_if)g + (s_if)\partial_i g$$ can be
extended, since $s_i$ and $\partial_i$ commute, to the identity: $$\partial^b(fg)\ = \ \sum_{b_1 \sqcup b_2 = b}(s^{b_2}\partial^{b_1}f)\partial^{b_2}g, $$  thus the following identity holds in  $\mathrm{BW}_n:$
$$y^bf\ = \ \underset{{b_1 \sqcup b_2 = b}} {\sum}s^{b_2}(\partial^{b_1}f)y^{b_2} \ \ \ \ \ \mbox{for}\ \ \ \   f\in \mathbb{Z}_2[x_1,...,x_n].$$ In particular we obtain that
$$y^bm^c \ = \ \sum_{b_1 \sqcup b_2 \subseteq b}s^{b_2}s^{b_1}(m^c)y^{b_2}\ = \
\sum_{b_1 \sqcup b_2 \subseteq b}m^{c+ b_1 + b_2}y^{b_2}\ = \ \sum_{b_1 \subseteq b_2 \subseteq b }m^{c+b_2}y^{b_1}.$$

\noindent 2. We have that:
$$m^ay^bm^cy^d\ = \ \sum_{b_1 \subseteq b_2 \subseteq b }m^a m^{c+b_2}y^{b_1} y^d\ = \
\sum_{b_1 \subseteq b_2 \subseteq b }\delta_{a,c+b_2}m^ay^{b_1 \sqcup d}$$
$$= \ \sum_{b_1 \subseteq a+c \subseteq b }m^ay^{b_1\sqcup d}\ = \ \underset{{\underset {e\setminus d \subseteq a+c \subseteq b }{d\subseteq e}}} {\sum}m^ay^{e}.$$
where the last identity follows from the fact that $b_2=a+c$ and $e=b_1 \sqcup d$. \\

\noindent 3. From the relations $y_ix_j=x_jy_i$ for $i\neq j$ and $y_ix_i = x_iy_i +y_i+1$ we can argue as follows.
If a letter $y_i$ is placed just to the left of a $x_j$ we can move it to the right, since these letters commute. If instead we have a product
$y_ix_i$, then three options arises:

\begin{description}
  \item[a)] $y_i$ moves to the right of $x_i$;
  \item[b)] $y_i$ absorbs $x_i$;
  \item[c)] $x_i$ and $y_i$ annihilate  each other
leaving an $1$.
\end{description}

Call $k_1$ the set of indices for which c) occurs, and
$k_2$ the set of indices for which either b) or c) occur. Then  $k_1 \subseteq k_2 \subseteq b \cap c$
and the set for which option a) occurs is $b \cap c \setminus k_2$. Thus the desired identity is obtained.\\

\noindent 4. We have that:
$$x^ay^bx^cy^d \ = \ \underset{{k_1 \subseteq k_2 \subseteq b \cap c}} {\sum}x^{a \cup c\setminus k_2}y^{(b\setminus k_1) \sqcup d}\ = \
\underset{{a \subseteq e,\ d \subseteq f}} {\sum}c(a, b, c, d, e, f)x^ey^f,$$ where $$c(a, b, c, d, e, f)\ = \
O\Big\{  k_1 \subseteq k_2 \subseteq b \cap c  \  | \  a \cup ( c \setminus k_2) = e ,\  \  b \setminus k_1= f \setminus d \Big\}.$$
\end{proof}

\begin{exmp}{\em $$y^{\{1\}}m^{\{1\}}=m^{\{1\}} + m^{\emptyset}+m^{\{\emptyset\}}y^{\{1\}}; \ \ \ \ \ \ \ \ \
m^{\{1\}}y^{\{1\}}m^{\{1\}}y^{\{1\}} = m^{\{1\}}y^{\{1\}}; $$
$$y^{\{1\}}m^{\{1,2\}}=m^{\{1,2\}} +m^{\{2\}}+m^{\{2\}}y^{\{1\}}; \ \ \ \ \ \  \ \ \ m^{\{2\}}y^{\{1\}}m^{\{1,2\}}y^{\{1\}}=m^{\{2\}}y^{\{1\}};$$
$$y^{\{1,2\}}m^{\{1,2,3\}}\ = \ m^{\{1,2,3\}} + m^{\{2,3\}} + m^{\{1,3\}}+m^{\{1\}} +m^{\{2,3\}}y^{\{1\}}+m^{\{3\}}y^{\{1\}} $$
$$+m^{\{1,3\}}y^{\{2\}}+m^{\{3\}}y^{\{2\}}+m^{\{3\}}y^{\{1,2\}}; \ \ \ \ \ \ \ \ \ m^{\{3\}}y^{\{1,2\}}m^{\{1,2,3\}}y^{\{1\}}= m^{\{3\}}y^{\{1,2\}}.$$}
\end{exmp}

\begin{exmp}{\em For $i \in [k]$ assume given $A_i \in \mathrm{P}\mathrm{P}[n]\ $ and $\ f_i= \sum_{a \in A_i} y^{a}$. Then
$$f_1\cdots f_k \ = \ \sum_{b \in  \mathrm{P}[n]} O\{(a_1,...,a_k)\in A_1\times ... \times A_k \ | \ a_1\sqcup \cdots \sqcup a_k = b\  \}y^{b}.$$
In particular, for $A \in \mathrm{P}\mathrm{P}[n]$ and $f=\sum_{a \in A}y^a$, we get
that $$f^k\ = \ \sum_{b \in \mathrm{P}[n]}O\{a_1,...,a_k\in A \ | \ a_1\sqcup \cdots \sqcup a_k = b\  \} y^{b}.$$
For example, if $A = \mathrm{P}[n]$ then for $k \geq 2$ we have that:
 $$f^k\ = \ \sum_{b \in \mathrm{P}[n]}O\{a_1,...,a_k\in \mathrm{P}[n] \ | \ a_1\sqcup  \cdots \sqcup a_{k} =b\  \} y^{b}=
 \sum_{b \in \mathrm{P}[n]}(k^{|b|} \ \mbox{mod} \ 2 ) \ y^{b},$$
thus $\ f^k=f \ $ if $k$ is odd and $\ f^k=1\ $ if $k$ is even.

}
\end{exmp}

From  Lemma \ref{bases} we see that there are several natural basis for  $\mathrm{BW}_n$, namely:
$$\{ m^ay^b\ | \ a,b \in \mathrm{P}[n] \ \}, \ \ \ \ \ \{ x^ay^b \ | \ a,b \in \mathrm{P}[n]  \ \}, \ \ \ \ \  \{ w^ay^b \ | \ a,b \in \mathrm{P}[n] \ \} .$$

We write the coordinates of $f \in\mathrm{ BW}_n$ in these bases as:
$$f\ =\ \sum_{a,b \in \mathrm{P}[n]} f_m(a,b) m^ay^b\ = \ \sum_{a,b \in \mathrm{P}[n]} f_x(a,b) x^ay^b \ = \ \sum_{a,b \in \mathrm{P}[n] } f_w(a,b) w^ay^b.$$

These coordinates systems are connected by the relations: $$f_x(b, c)=\sum_{a \subseteq b} f_m(a, c), \ \ \ \ \ f_m(b, c)=\sum_{a \subseteq b} f_x(a, c),
\ \ \ \ \  f_w(b, c)=\sum_{a \subseteq b} f_m(\overline{a}, c),$$
$$f_m(b,c)=\sum_{a \subseteq \overline{b}} f_w(a, c), \ \ \ \ \  f_x(a, c)=\sum_{a \subseteq b} f_w(b, c) , \ \ \ \ \ f_w(a)=\sum_{a \subseteq b} f_x(b, c).$$

\begin{thm}\label{tt1}{\em For $f,g \in\mathrm{ BW}_n$ the following identities hold for $a,e,h \in \mathrm{P}[n]:$
\begin{enumerate}
\item $ \ \ (fg)_m(a,e)\ = \ \underset{{\underset {e\setminus d \subseteq a+c \subseteq b }{ b, c,  d \subseteq e}}} {\sum}f_{m}(a,b)g_{m}(c,d).$
\item $ \ \ (fg)_x(e,h)\ = \ \underset{{a \subseteq e, b,  c,  d \subseteq h}} {\sum}c(a,b,c,d,e,h)f_{x}(a,b)g_{x}(c,d), \ \mbox{where}$
$$c(a,b,c,d,e,h)\ = \ O\Big\{ k_1 \subseteq k_2 \subseteq b \cap c \ \  | \ \ a \cup ( c \setminus k_2) = e ,\  \  b \setminus k_1= h \setminus d \Big\}.$$
\end{enumerate}
}
\end{thm}

\begin{proof}
1. \ Let $\ f=\sum_{a,b \in \mathrm{P}[n]}f_{m}(a,b)m^ay^b, \ $ $\ g=\sum_{c,d \in \mathrm{P}[n]}g_{m}(c,d)m^cy^d,$  then we have that:
$$ fg\ = \ \sum_{a,b,c,d \in \mathrm{P}[n]}f_{m}(a,b)g_{m}(c,d) m^ay^bm^cy^d$$
$$= \ \sum_{a,b,c,d,e \in \mathrm{P}[n]}f_{m}(a,b)g_{m}(c,d)\underset{{\underset {e\setminus d \subseteq a+c \subseteq b }{d\subseteq e}}} {\sum}m^ay^{e}$$
$$= \ \sum_{ d \subseteq e, \ e\setminus d \subseteq a+c \subseteq b }f_{m}(a,b)g_{m}(c,d)m^ay^{e}.$$

\noindent 2. \ Let $f=\sum_{a,b \in \mathrm{P}[n]}f_{x}(a,b)x^ay^b\ $ and $\ g=\sum_{c,d \in \mathrm{P}[n]}g_{x}(c,d)x^cy^d, \ $  then we have that:
$$fg \ = \ \sum_{a,b,c,d \in \mathrm{P}[n]}f_{x}(a,b)g_{x}(c,d)x^ay^bx^cy^d$$
$$= \ \sum_{a,b,c,d \in \mathrm{P}[n]} \underset{{a \subseteq e,\ d \subseteq h}} {\sum}f_{x}(a,b)g_{x}(c,d)c(a,b,c,d,e,h)x^ey^h,$$
where $$c(a,b,c,d,e,h)\ = \ O\Big\{ k_1 \subseteq k_2 \subseteq b \cap c \ | \ a \cup ( c \setminus k_2) = e ,\  \  b \setminus k_1= h \setminus d \Big\}.$$
\end{proof}

\begin{exmp}{\em Let $\ x^ry^r= \sum_{a,b \in \mathrm{P}[n]}f_m(a,b)m^ay^b\ $ and $\ x^sy^s= \sum_{a,b \in \mathrm{P}[n]}g_m(a,b)m^ay^b$. Then
$$(fg)_m^2(a,e)\ = \ \underset{{\underset {e\setminus d \subseteq a+c \subseteq b }{ b, c,  d \subseteq e}}} {\sum}f_{m}(a,b)g_{m}(c,d).$$
For a non-vanishing summand we must have that $a=b=r$, $c=d=s$, and $s\subseteq e$. The conditions $e\setminus s \subseteq r+s \subseteq r$
implies that $s \subseteq r$ and $e\setminus s \subseteq r\setminus s$, thus $e \subseteq r $. We conclude that $(fg)_m^2(a,e)=1$ iff $s \subseteq r$, $a=r$ and $s \subseteq e \subseteq r.$ Thus $x^ry^rx^sy^s=0$ if $s\nsubseteq r$. For $s\subseteq r$ we get
$$x^ry^rx^sy^s \ = \ \sum_{s \subseteq e \subseteq r}x^ry^e.$$
In particular we get that $(x^ry^r)^n=x^ry^r$.}
\end{exmp}

\begin{exmp}{\em Let $\ f \ = \ \sum_{a,b \in \mathrm{P}[n]}m^ay^b\ = \ \sum_{a,b \in \mathrm{P}[n]}f_m(a,b)m^ay^b. \ $ We have that:
$$f_m^2(a,e)\ = \ \underset{{\underset {e\setminus d \subseteq a+c \subseteq b }{ b, c,  d \subseteq e}}} {\sum}1
\ \ = \ \ O\Big\{b,c,d \ | \ d \subseteq e, \ e\setminus d \subseteq a+c \subseteq b \Big\}.$$
Note that if $a+c$ is not equal to $[n]$, then there are an even number of choices for $b$, thus we
can assume that $c=\overline{a}$ and $b=[n].$ The condition $e\setminus d \subseteq a+c=[n]$ becomes
trivial, and therefore $f^2(a,e)= O\mathrm{P}[|e|]=0$ if
$e \neq \emptyset$ and  $f^2(a,e)= 1$ if $e = \emptyset$. Therefore we have that
$$f^2 \ = \ \sum_{a\in \mathrm{P}[n]}m^a.$$

}
\end{exmp}

\begin{exmp}{\em Let $\ r \ = \ \sum_{i \in[n]}x^{\{i\}}y^{\{i\}}\ = \ \sum_{a,b \in \mathrm{P}[n]} r_x(a,b) x^ay^b\in \mathrm{ BW}_n, \ $ then we have that:
$$r_x^2(e,f) \ = \  \underset{{a \subseteq e, b,  c,  d \subseteq f}} {\sum}c(a,b,c,d,e,f)r_{x}(a,b)r_{x}(c,d), \ \ \ \ \   \mbox{where}$$
$$c(a,b,c,d,e,f)\ = \ O\Big\{ k_1 \subseteq k_2 \subseteq b \cap c \ \  | \ \ a \cup ( c \setminus k_2) = e ,\  \  b \setminus k_1= f \setminus d \Big\}.$$
Clearly $|a|=|b|=|c|=|d|=1,$ $a=b,$ and $c=d.$ Moreover, we have that $|b \cap c| \leq 1$.
If $|b \cap c| =1$, then $a=b=c=d=e=\{ i\}$ for some $i \in [n].$ If $k_1= \emptyset$, then there are two options for
$k_2$ leading to a vanishing coefficient. Thus we may assume that $k_1 = k_2 = \{i\}$ and then necessarily
$f=\{i\}$. Thus we conclude that $r_x^2(\{ i\}, \{ i\})=1.$
If instead $|b \cap c|=0$, then $k_1 = k_2= \emptyset$,  $a \cup c = e$ and $b = f \setminus d.$
Let $i \neq j$ and suppose that $a=b= \{i\}$ and $c=d=\{j\}$. Then $e=f=\{i,j\}$ and $r_x^2(\{i,j\}, \{i,j\})=1$.
All together we conclude that
$$r^2\ = \ \sum_{i \in[n]}x^{\{i\}}y^{\{i\}} \ + \ \sum_{i\neq j \in[n]}x^{\{i,j\}}y^{\{i,j\}}.$$
}
\end{exmp}

\begin{exmp}{\em From Corollary \ref{coro} we see that if $\ s= \sum_{i \in[n]}w^{\{i\}}y^{\{i\}} \ $ then
$$s^2\ = \ \sum_{i \in[n]}w^{\{i\}}y^{\{i\}} \ + \ \sum_{i\neq j \in[n]}w^{\{i,j\}}y^{\{i,j\}}.$$
Equivalently, if $s\ = \ \sum_{i \in[n]}y^{\{i\}} \ + \ \sum_{i \in[n]}x^{\{i\}}y^{\{i\}} $ then
$$s^2\ = \ \sum_{i \in[n]}y^{\{i\}} + \sum_{i \in[n]}x^{\{i\}}y^{\{i\}}+ \sum_{i\neq j \in[n]}y^{\{i,j\}}
+ \sum_{i\neq j \in[n]}x^{\{i, j\}}y^{\{i,j\}} + \sum_{i\neq j \in[n]}\left(x^{\{i\}} + x^{\{j\}}\right)y^{\{i,j\}}.$$  }
\end{exmp}

\section{A Shifted Presentation}\label{s3}

So far, the operators  $\partial_i$ have played the main role. In this section we take an alternative viewpoint
and let the operators $s_i$ be the main characters. Recall that the Boolean partial derivatives and the Boolean shift operators are related by the
identities $y_i = s_i +1$ and $s_i = y_i +1$. For $a, b \in \mathrm{P} [n], \ $ set
$\ y^{a}=\underset{{i\in a}} {\prod}y_i \ \ $ and $\ \ s^{a}=\underset{{i\in a}} {\prod}s_i. \ \ $ We have that:
$$y^{b}\ = \ \prod_{i\in b}y_i=\prod_{i\in b}(s_i + 1)\ = \ \sum_{a \subseteq b}s^a, \ \ \mbox{and by the M} \ddot{\mbox{o}} \mbox{bius inversion
formula} \ \ s^{b}= \sum_{a \subseteq b}y^a.$$

\begin{prop}\label{dos}
{\em Consider maps $\ D:\mathrm{P}[n]\times \mathrm{P}[n] \longrightarrow \mathbb{Z}_2\ $ and $\ f:\mathrm{P}[n] \longrightarrow \mathbb{Z}_2$.\\

\noindent 1. Let $D=\underset{{a,b \in \mathrm{P}[n]}} {\sum}D(a,b)m^a s^b \in \mathrm{BDO}_n$,
$f=\underset{{c \in \mathrm{P} [n]}} {\sum}f(c)m^c \in \mathbb{Z}_2[\mathbb{A}^n]$, and
$Df=\underset{{a \in \mathrm{P}[n]}} {\sum}Df(a)m^a $. Then we have that $$Df(a)\ = \ {\underset{ b \in \mathrm{P}[n]}{\sum}}D(a,b)f(a+b). $$
\noindent 2. Let $D=\underset{{a,b \in \mathrm{P} [n]}} {\sum}D_x(a,b)x^as^b \in \mathrm{BDO}_n$, \
$f=\underset{{c \in \mathrm{P} [n]}} {\sum}f_x(c)x^c \in \mathbb{Z}_2[\mathbb{A}^n]$, \ and
$Df=\underset{{d \in \mathrm{P} [n]}} {\sum}Df_x(d)x^d$. Then
$$Df_x(d)\ = \ \underset{{\underset {a \cup (c\setminus e)=d}{a,\ e \subseteq  b \cap c}}} {\sum} D_x(a,b)f_x(c) .$$
}
\end{prop}

\begin{proof}
\ \\
\ \\
\noindent 1.  $Df\ = \ \underset{{a,b,c  \in \mathrm{P}[n]}} {\sum}D(a,b)f(c)m^as^bm^c\ = \
\underset{{a, b,c  }} {\sum}D(a,b)f(c)m^am^{b+c}\ = \ \underset{{a, b }} {\sum}D(a,b)f(a+b)m^a$ \\

$\ \ \ \  \ = \ \underset{a  \in \mathrm{P}[n]}{\sum} \left( \underset{b \in \mathrm{P}[n]}{\sum }D(a,b)f(a+b) \right)m^a.$ \\

\noindent 2. $Df \ = \ \underset{{a,b,c  \in \mathrm{P}[n]}} {\sum}D_x(a,b)f_x(c)x^as^bx^c\ = \
\underset{{a, e \subseteq  b \cap c  }} {\sum}D_x(a,b)f_x(c)x^{a \cup c \setminus e}$ \\

$\ \ \ \ \ =  \ \underset{d  \in \mathrm{P}[n]} {\sum}  \left(  \underset{{\underset {a \cup (c\setminus e)=d}{a, \ e \subseteq  b \cap c}}} {\sum} D_x(a,b)f_x(c)  \right)x^d.$

\end{proof}

Proposition \ref{dos} and Theorem \ref{egre} provide a couple of explicit ways of identifying $\mathrm{BDO}_n$ with
$\mathrm{M}_{2^n}(\mathbb{Z}_2)$ the algebra of square matrices of size $2^n$ with coefficients in $\mathbb{Z}_2.$
Consider the $\mathbb{Z}_2$-linear map $$\mathrm{R}:\mathrm{BDO}_n \ \longrightarrow \ \mathrm{M}_{2^n}(\mathbb{Z}_2)$$ sending
$m^as^b$ to $\mathrm{R}(m^as^b)$ the matrix of $m^as^b$ on the basis $m^a$.
The action of $m^as^b$ on $m^c$ is given by $m^as^bm^c=m^a$  if $c=a+b$ and $0$ otherwise.\\

Therefore, the matrix $\mathrm{R}(m^as^b)$ is given for $c,d \in \mathrm{P}[n]$ by the rule
$$\mathrm{R}(m^as^b)_{c,d}\ = \
\left\{\begin{array}{cc} 1 & \mathrm{if}\  c=a \ \mbox{and }\ d=a + b \\
0 & \ \mathrm{otherwise}  \end{array}\right. $$

\begin{exmp}{\em The graph of the matrix $\mathrm{R}(m^as^b)_{c,d}$ is shown in Figure \ref{g3}.
}
\end{exmp}

\begin{figure}[ht]
\begin{center}
\includegraphics[height=1cm]{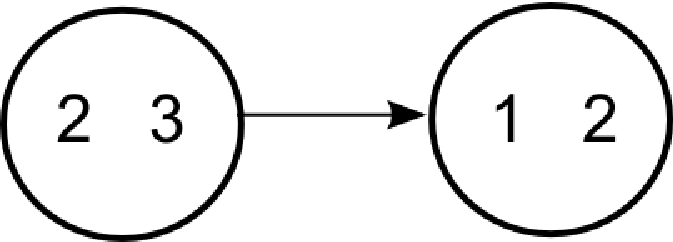}
\caption{ \ Graph of the matrix $\mathrm{R}(m^{\{1,2\}}\partial^{\{2,3 \}})$. \label{g3}}
\end{center}
\end{figure}

For a second representation consider $\mathbb{Z}_2$-linear
the map $\mathrm{S}:\mathrm{BDO}_n \rightarrow  \mathrm{M}_{2^n}(\mathbb{Z}_2)$ sending $x^as^b$
to $\mathrm{S}(x^as^b),$ the matrix of $x^as^b$ on the basis $x^a$.
The action of $x^as^b$ on $x^c$ is given by $$x^as^bx^c\ = \ \sum_{e \subseteq b \cap c}x^{a \cup c\setminus e}.$$
Therefore, the matrix $\mathrm{S}(x^as^b)$ is given for $c,d \in \mathrm{P}[n]$ by the rule
    $$\mathrm{S}(x^as^b)_{c,d} \ = \
\left\{\begin{array}{cc} 1 & \mathrm{if}\ \ \ O\{e \subseteq b\cap d \ | \ c = a \cup d\setminus e  \}\\
0 & \ \mathrm{otherwise}  \end{array}\right. $$

\begin{exmp}{\em
The graph of the matrix $\mathrm{S}(x^{\{1,2\}}s^{\{1,3 \}})$ is shown in Figure \ref{g4}.
}
\end{exmp}

\begin{figure}[ht]
\begin{center}
\includegraphics[height=2.3cm]{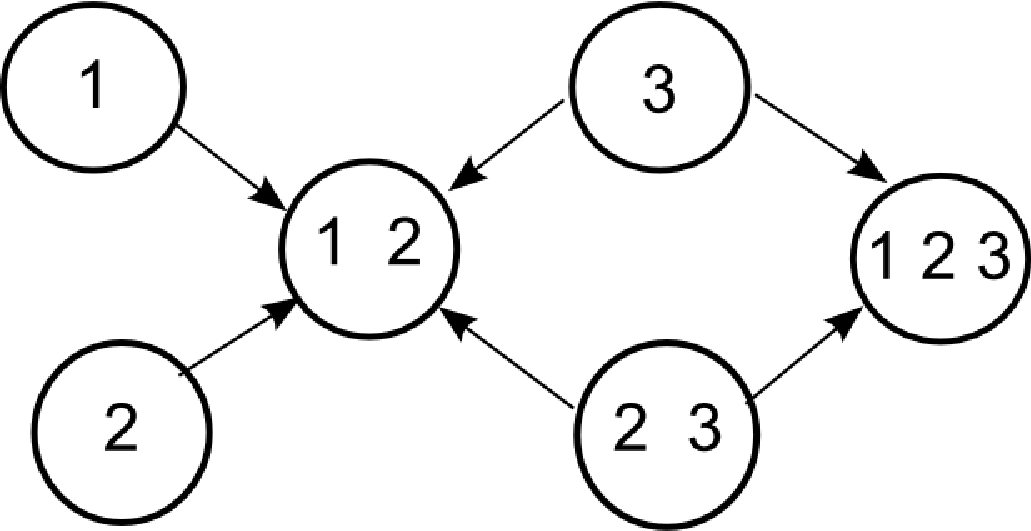}
\caption{ \ Graph of the matrix $\mathrm{R}(m^{\{1,2\}}\partial^{\{2,3 \}})$. \label{g4}}
\end{center}
\end{figure}

Next we introduce the shifted Boole-Weyl algebra $\mathrm{SBW}_n$, a quite useful and easy to handle presentation
for the algebra of Boolean differential operators $\mathrm{BDO}_n$.

\begin{defn}
{\em The algebra $\mathrm{SBW}_n$ is the quotient of $\mathbb{Z}_2<x_1,...,x_n, s_1,...,s_n>$,
the free associative $\mathbb{Z}_2$-algebra generated by $x_1,...,x_n, s_1,...,s_n$, by the ideal
$$<x_i^2 +x_i, \ x_ix_j + x_jx_i, \ s_is_j + s_js_i, \ s_i^2 +1, \ s_ix_j + x_js_i, \ s_ix_i + x_is_i + s_i>,  $$
generated by the relations $x_i^2 =x_i, \ s_i^2=1,$ and $s_ix_i = x_is_i + s_i$ for $i \in [n]$,
$x_ix_j = x_jx_i$ and $y_iy_j = y_jy_i$ for $i,j \in [n],$ and $s_ix_j + x_js_i$ for $i \neq j \in [n].$
}
\end{defn}

\begin{thm}{\em The map $ \mathbb{Z}_2<x_1,...,x_n, s_1,...,s_n> \ \longrightarrow \ \mathrm{End}_{\mathbb{Z}_2}(\mathbb{Z}_2[\mathbb{A}^n])$ sending $x_i$ to the operator of multiplication by $x_i$, and $s_i$ to the shift operator in the $i$-direction, descends to an isomorphism   $ \mathrm{SBW}_n \longrightarrow \mathrm{End}_{\mathbb{Z}_2}(\mathbb{Z}_2[\mathbb{A}^n])$  of $\mathbb{Z}_2$-algebras.}
\end{thm}

\begin{proof}
One can check that the map $ \mathbb{Z}_2<x_1,...,x_n, s_1,...,s_n> \ \longrightarrow \ \mathbb{Z}_2<x_1,...,x_n, y_1,...,y_n>$ sending
$x_i$ to $x_i$ and $s_i$ to $y_i+1$ descends to an algebra isomorphisms $\mathrm{SBW}_n \longrightarrow \mathrm{BW}_n$.
The result then follows from Theorem \ref{qbas}.
\end{proof}

\begin{thm}{\em The map $ \mathbb{Z}_2<x_1,...,x_n, s_1,...,s_n> \ \longrightarrow \
\mathrm{End}_{\mathbb{Z}_2}(\mathbb{Z}_2[\mathbb{A}^n])$ sending $x_i$ to the operator of multiplication by $w_i=x_i +1$, and $s_i$ to the shift operator in the $i$-direction, descends to an isomorphism  $ \mathrm{SBW}_n \ \longrightarrow \ \mathrm{End}_{\mathbb{Z}_2}(\mathbb{Z}_2[\mathbb{A}^n])$  of $\mathbb{Z}_2$-algebras.}
\end{thm}

\begin{proof}
Follows from the fact that $w_i$ and $s_j$ satisfy exactly the same relation as $x_i$ and $s_j.$
\end{proof}

\begin{cor}{\em Any identity in $\mathrm{SBW}_n$ involving $x_i$ and $s_j$  has an associated
identity involving $w_i$ and $s_j$ obtained by replacing $x_i$ by $w_i.$}
\end{cor}

\begin{lem}{\em For $a,b, c, d \in \mathrm{P}[n]$ the following identities hold in $\mathrm{SBW}_n$ :
$$1. \  s^bm^c\ = \ m^{b+c}s^b. \ \ \ \ \ \ \ \ \ \ \ \ \ \ \ \ 2. \ m^as^bm^cs^d\ = \ \delta_{a, b+c}m^as^{b+d}.\ \ $$
$$3. \ s^bx^c\ = \ \sum_{k \subseteq b\cap c} x^{c\setminus k}s^b.  \ \ \ \ \ \ \ \ \ \ \ 4. \ x^as^bx^cs^d \ = \ \sum_{e \subseteq b\cap c} x^{a \cup c\setminus e}s^{b+d}.$$
}
\end{lem}
\begin{proof}
1. For any $f \in \mathbb{Z}_2[\mathbb{A}^n]$ we have that:  $$(s^b m^c f)(x)\ =\ m^c(x +b)f(x+b)\ = \ m^{b+c}(x)f(x+b)\ = \ m^{b+c}s^bf(x),
\ \mbox{thus} \  s^b m^c \ = \ m^{b+c}s^b.$$

\noindent 2. $m^as^bm^cs^d\ = \ m^am^{b+c}s^bs^d\ = \ \delta_{a, b+c}m^{b+c}s^{b+d}.$\\

\noindent 3. From the identity $s_ix_i=x_is_i +s_i$ we see that as $s_i$ pass to the right of $x_i$, it may or may not absorb $x_i$. The set $k\subseteq b\cap c$
is the set of indices for which $x_i$ is absorbed by $s_i$.\\

\noindent 4. $x^as^bx^cs^d \ = \ \sum_{e \subseteq b\cap c} x^ax^{c\setminus e}s^bs^d\ = \ \sum_{e \subseteq b\cap c} x^{a \cup c\setminus e}s^{b+d}.$

\end{proof}

\begin{exmp}{\em } $$1. \  s^{[n]}m^c\ = \ m^{\overline{c}}s^{[n]}, \ \ \ \overline{c}=[n]\setminus c.   \ \  \ \ \ \ \ \ \ \ \
2. \ m^{\overline{c}}s^{[n]}m^cs^d \ = \ m^{\overline{c}}s^{\overline{d}}.$$
$$3. \ s^{[n]}x^c\ = \ \sum_{k \subseteq  c} x^{k}s^{[n]}. \ \ \ \ \ \ \ \ \ \ \ \ \ \ \ \ \ \ \ \
4. \ x^as^{[n]}x^cs^d \ = \ \sum_{k \subseteq  c} x^{a \cup k}s^{\overline{d}}.$$
\end{exmp}

From  Lemma \ref{bases} we see that there are several natural basis for  $\mathrm{BW}_n$, namely:
$$\{ m^as^b\ | \ a,b \in \mathrm{P}[n] \ \}, \ \ \ \ \ \{ x^as^b \ | \ a,b \in \mathrm{P}[n]  \ \}, \ \ \ \ \ \{ w^as^b \ | \ a,b \in \mathrm{P}[n]\ \} .$$

We write the coordinates of $f \in\mathrm{SBW}_n$ in these bases as:
$$f\ = \ \sum_{a,b \in \mathrm{P}[n] } f_{m,s}(a,b) m^as^b \ = \ \sum_{a,b \in \mathrm{P}[n] } f_{x,s}(a,b) x^as^b \ = \ \sum_{a,b \in \mathrm{P}[n]} f_{w,s}(a,b) w^as^b.$$

These coordinates systems are connected by the relations: $$f_{x,s}(b, c)\ = \ \sum_{a \subseteq b} f_{m,s}(a, c), \ \ \ \ f_{m,s}(b, c)\ = \ \sum_{a \subseteq b} f_{x,s}(a, c),
\ \ \ \ f_{w,s}(b, c)\ = \ \sum_{a \subseteq b} f_{m,s}(\overline{a}, c),$$
$$f_{m,s}(b,c)\ = \ \sum_{a \subseteq \overline{b}} f_{w,s}(a, c), \ \ \ \ f_{x,s}(a, c)\ = \ \sum_{a \subseteq b} f_{w,s}(b, c) , \ \ \ \ f_{w,s}(a)\ = \ \sum_{a \subseteq b} f_{x,s}(b, c).$$

\begin{thm}\label{tt2}{\em For $f,g \in \mathrm{SBW}_n$ the following identities hold for $a,b,e,h \in \mathrm{P}[n]$:
\begin{enumerate}
\item $(fg)_{m,s}(a,b)\ = \ \underset{{c \in \mathrm{P}[n]}} {\sum}f_{m,s}(a,c)g_{m,s}(a+c,b+c).$
\item $(fg)_{x,s}(e,h)\ = \ \sum_{a \subseteq e, \ b,c \in \mathrm{P}[n] }
O\Big\{k \subseteq  b\cap c \ | \ a \cup c\setminus k =e  \ \Big\}f_{x,s}(a,b)g_{x,s}(c,b+h).$
\end{enumerate}
}
\end{thm}

\begin{proof}
1. Let $f=\sum_{a,b \in \mathrm{P}[n]}f_{m,s}(a,b)m^as^b,$ $\ g=\sum_{c,d \in \mathrm{P}[n]}g_{m,s}(c,d)m^cs^d,$  then we have that:
$$ fg\ = \ \sum_{a,b, c,d \in \mathrm{P}[n]}f_{m,s}(a,b)g_{m,s}(c,d) m^as^bm^cs^d$$
$$= \ \sum_{b,c,d \in \mathrm{P}[n]} f_{m,s}(b+c,b)g_{m,s}(c,d) m^{b+c}s^{b+d}$$
$$= \ \sum_{e,f \in \mathrm{P}[n]} \left( \underset{{\underset {b+d=f }{b+c=e}}} {\sum}f_{m,s}(b+c,b)g_{m,s}(c,d)\right) m^{e}s^{f}$$
$$= \ \sum_{e,f \in \mathrm{P}[n]} \left( \underset{{b \in \mathrm{P}[n]}} {\sum}f_{m,s}(e,b)g_{m,s}(e+b,f+b)\right) m^{e}s^{f}.$$

\

\noindent 2. Let $f=\sum_{a,b \in \mathrm{P}[n]}f_{x,s}(a,b)x^as^b, \ $ $\ g=\sum_{c,d \in \mathrm{P}[n]}g_{x,s}(c,d)x^cs^d,$  then we have that:
$$fg\ = \ \sum_{a,b, c,d \in \mathrm{P}[n]} f_{x,s}(a,b)g_{x,s}(c,d) x^as^bx^cs^d$$
$$= \ \sum_{a,b,c,d \in \mathrm{P}[n], \ k \subseteq b\cap c} f_{x,s}(a,b)g_{x,s}(c,d) x^{a \cup c\setminus k}s^{b+d}$$
$$= \ \sum_{e,h \in \mathrm{P}[n]}\left(\underset{ { \underset {\underset{a \cup c\setminus k =e, \ b+d=h } {k \subseteq b\cap c} }{a,b,c,d \in \mathrm{P}[n]} } } {\sum} f_{x,s}(a,b)g_{x,s}(c,d) \right) x^{e}s^{h}$$
$$= \ \sum_{e,h \in \mathrm{P}[n]}\left(\underset{ { \underset {\underset{a \cup c\setminus k =e } {k \subseteq b\cap c} }{a \subseteq e, \ b,c \in \mathrm{P}[n]} } } {\sum} f_{x,s}(a,b)g_{x,s}(c,b+h) \right) x^{e}s^{h}$$
$$= \ \sum_{e,h \in \mathrm{P}[n]}\left(\sum_{a \subseteq e, \ b,c \in \mathrm{P}[n] }
O\Big\{k \subseteq  b\cap c \ | \ a \cup c\setminus k =e  \ \Big\}f_{x,s}(a,b)g_{x,s}(c,b+h) \right) x^{e}s^{h}.$$
\end{proof}

\begin{exmp}{\em  Suppose that $f=\sum_{a,b \in \mathrm{P}[n]}f_{m,s}(a,b)m^as^b\ $ and  $\ \ g=\sum_{c,d \in \mathrm{P}[n]}g_{m,s}(c,d)m^cs^d$
are actually regular functions on $\mathbb{Z}_2^n$, i.e. $f_{m,s}(a,b)=0$ if $b \neq \emptyset$, and $g_{m,s}(c,d)=0$ if $d \neq \emptyset$.
A non-vanishing term in the expression $$(fg)_{m,s}(a,b)\ = \ \underset{{c \in \mathrm{P}[n]}} {\sum}f_{m,s}(a,c)g_{m,s}(a+c,b+c)$$
must have $c = \emptyset$, and then we must also have that $c= \emptyset + c=\emptyset$, and $a+c=a+\emptyset=a.$ Thus in this case the product $fg$
is, as expected, just the pointwise product of  functions on $\mathbb{Z}_2^n$.
}
\end{exmp}

\begin{exmp}{\em  Let $f=\sum_{a,b \in \mathrm{P}[n]}f_{m,s}(a,b)m^as^b,$ and  suppose that $g=\sum_{c,d \in \mathrm{P}[n]}g_{m,s}(c,d)m^cs^d$
is such that $g_{m,s}(c,d)=0$ if $c\neq [n].$ Then a non-vanishing summand in the formula
$$(fg)_{m,s}(a,b)\ = \ \underset{{c \in \mathrm{P}[n]}} {\sum}f_{m,s}(a,c)g_{m,s}(a+c,b+c)$$
can only arise for $c=\overline{a}$. Therefore $(fg)_{m,s}(a,b)= f_{m,s}(a,\overline{a})g_{m,s}([n],b+\overline{a}).$
For example, we have that
$$\left( \sum_{a \in \mathrm{P}[n]}m^as^{\overline{a}} \right)\left( \sum_{d \in \mathrm{P}[n]}m^{[n]}s^d \right)\ = \ \sum_{a,b \in \mathrm{P}[n]}m^as^b.$$ As another example consider $f=\sum_{a,b \in \mathrm{P}[n]}m^as^b$
and $g=m^{[n]}s^{[n]}$. In this case we get that:
$$\left(\sum_{a,b \in \mathrm{P}[n]}m^as^b\right)\left( m^{[n]}s^{[n]} \right)\ = \ \sum_{a\in \mathrm{P}[n]}m^as^a.$$
}
\end{exmp}

\section{Quantum Operational Logic}\label{s5}

In this section we study quantum Boolean algebras from a logical viewpoint. Propositional logic may be approach from a myriad of viewpoints. Here we take a revisionist approach bias towards the theory of operads and props. We believe this approach may be of interest in itself, and is certainly pretty convenient for our current purposes as it would readily generalize to cover quantum operational logic. We assume the
reader to be familiar with the language of operads and props \cite{boa, DR, g, markl, m}. First we review
the basic principles of classical propositional logic \cite{bro} which may be summarized as:

\begin{itemize}

\item On the syntactic side,  propositions are words in a certain language. Propositions are either simple or composite.
Let $x$ be the finite set of simple propositions, and $\mathbb{P}(x)$ be the
set of all propositions. Composite propositions are obtained from the simple propositions using the logical connectives. There are several
options for the choice of connectives, the most common ones being $\vee, \wedge, \rightarrow, \neg.$

\item On the semantics side, a truth function $\widehat{p}:\mathbb{Z}_2^x  \longrightarrow  \mathbb{Z}_2$ is associated to each proposition $p \in \mathbb{P}(x)$, where
$\mathbb{Z}_2^x$ is the set of maps from $x$ to $\mathbb{Z}_2$.
The map $$\mathbb{P}(x) \ \longrightarrow \ \mathbb{Z}_2[\mathbb{A}^x]$$
sending a proposition $p$ to its truth function $\ \widehat{p}\in \mathbb{Z}_2[\mathbb{A}^x]=\mathrm{M}(\mathbb{Z}_2^x,\mathbb{Z}_2)\ $ is such that:

\begin{itemize}
  \item $\widehat{a}\ $ is evaluation at $a, \ $ i.e. $\ \widehat{a}f=f(a)\ $ for $\ a \in x,\ $ and $\ f\in \mathbb{Z}_2^x. $
  \item $\widehat{p \vee q} = \widehat{p} \vee \widehat{p}, \ \ \ \ \ \ \
\widehat{p \wedge q} \ = \ \widehat{p} \wedge \widehat{p}, \ \ \ \ \ \ \ \widehat{p \rightarrow q} \ = \ \widehat{p} \rightarrow \widehat{p}, \ \ \ \ \ \ \ \widehat{ \neg p } \ = \ \neg \widehat{p},$\\
 where the action of the connectives on truth functions comes from the corresponding
operations on $\mathbb{Z}_2.$
\end{itemize}

\end{itemize}

The map $\ \mathbb{P}(x) \longrightarrow \mathbb{Z}_2[\mathbb{A}^x]\ $ is surjective, and there is a systematic procedure
to tell when two propositions have the same associated truth function. \\

For our purposes, it is convenient to
describe $\mathbb{P}(x)$ using the binary connectives product $.$ and sum $+$, and the constants $0,1.$ \\

In  logical terms  the product $.$ is the
logical conjunction,  $+$ is the exclusive or, and $0$ and $1$ represent falsity and truth, respectively.\\

$\mathbb{P}(x)$ is defined recursively as the set of words in the symbols
$a \in x,., +,0,1,(,)$  such that:
\begin{itemize}
\item $x\subseteq \mathbb{P}(x), \ $ $0 \in \mathbb{P}(x), \ $ and $\ 1 \in \mathbb{P}(x).$
\item If $p,q \in \mathbb{P}(x),$ then $(pq)\ $ and $\ (p+q)$ are also in $\mathbb{P}(x)$.
\end{itemize}

We defined recursively the notion of sub-words in $\mathbb{P}(x)$. For all $p,q,r \in \mathbb{P}(x)$ set:
\begin{itemize}
  \item $p$ is a sub-word of $p$;
  \item $p$ is a sub-word of $(pq)$ and $(p+q)$;
  \item if $p$ is a sub-word of $q$ and $q$ is a sub-word $r$, then $p$ is a sub-word of $r$.
\end{itemize}

Next we define an equivalence relation $\mathbb{R}(x)$, also denoted by $\sim$, on $\mathbb{P}(x)$. Given $p,q \in \mathbb{P}(x)$ we set:
 $$p \ \mathbb{R}(x) \ q \ \ \ \ \mbox{if and only if} \ \ \ \ \widehat{p}=\widehat{q}.$$ The relation
$ \mathbb{R}(x)$ can be defined in syntactic terms as follows. Propositions  $p$ and $q$ are related if and only if either $p=q$ or there exists a sequence
$p_1,...,p_k$, for some $k \geq 1$,  such that $p_1=p$, and $p_k=q$, and $p_{i+1}$ is obtained from $p_i$ by replacing
a sub-word of $p_i$ by an equivalent word according to the following relations valid for all propositions $p,q,r \in \mathbb{P}(x):$

\begin{itemize}
\item Associativity and commutativity for $.$ and $+$:
$$ \ p(qr) \ \sim \ (pq)r,\ \ \ \ \ pq\ \sim \ qp, \ \ \ \ \ (p+q)+r \ \sim \ p+(q+r), \ \ \ \ \  p+q \ \sim \ q+p.$$
\item Distributivity: $\ p(q+r)\ \sim \ pq+pr.$
\item Additive and multiplicative units:  $\ 0+ p \ \sim \ p \ \ \ $ and  $\ \ \ 1p \ \sim \ p$.
\item Additive nilpotency: $\ p+p \  \sim \ 0$.
\item Multiplicative idempotency: $\ pp \ \sim \ p$.
\end{itemize}

Let $\mathrm{Set}$ be the category of sets, and $\mathrm{set}$ be the full subcategory of finite sets. Let
$$\mathbb{Z}_2^{(\ )}: \mathrm{set}^{\circ}\ \longrightarrow \ \mathrm{set}$$
be the functor sending a finite set $x$ to the free Boolean algebra generated by $x$, i.e.
$$\mathbb{Z}_2^x \ \simeq \  \mathrm{P}(x);$$
and sending a map $\ f:x \longrightarrow y\ $ to the map $ \ \mathbb{Z}_2^y \longrightarrow \mathbb{Z}_2^x\ $ sending $\ g \in \mathbb{Z}_2^y\ $ to $\ g \circ f$. Let $$\mathbb{Z}_2[\mathbb{A}^{(\ )}] :  \mathrm{set} \ \longrightarrow \ \mathrm{set}$$ be the functor given by
$$\mathbb{Z}_2[\mathbb{A}^{(\ )}] \ = \ \mathbb{Z}_2^{(\ )} \circ \mathbb{Z}_2^{(\ )} \ = \ \mathbb{Z}_2^{\mathbb{Z}_2^{(\ )}},$$
i.e. $\mathbb{Z}_2[\mathbb{A}^{x}]$ is the algebra of regular Boolean functions on the affine space $\mathbb{Z}_2^x.$ \\

Recall that an operad $O\ $ in $\ \mathrm{Set}$ is given by a sequence of sets $\{O(n)\}_{n\in \mathbb{N}}$, together  with right actions of the permutations groups $$O(n) \times S_n \ \ \longrightarrow \ \ O(n)$$ and composition maps
$$c_k:O(k)\times O(n_1) \times \cdots \times O(n_k)\ \longrightarrow \ O(n_1 + \cdots n_k)$$
which satisfy the equivariance, associativity, and unity axioms \cite{m}. \\

Any set $X$ determines the endomorphism operad $\mathrm{End}_X$  defined by the sequence
$$\{\mathrm{End}_X(n)\}_{n\in \mathbb{N}} \ \ = \ \ \{\mathrm{M}(X^n,X)\}_{n\in \mathbb{N}}.$$ A pertumation $\alpha \in S_n$ acts on a map $f: X^n \longrightarrow X$ by
$$f\alpha(x_1,...,x_n) \ = \ f(x_{\alpha^{-1}1},...,x_{\alpha^{-1}n}) .$$  The composition maps arise as follows:
$$\mathrm{M}(X^k,X)\times \mathrm{M}(X^{n_1},X)\times \cdots \times \mathrm{M}(X^{n_1},X) \ \ \simeq $$
$$\mathrm{M}(X^k,X)\times \mathrm{M}(X^{n_1+ \cdots + n_k},X^k) \ \ \longrightarrow \ \ \mathrm{M}(X^{n_1+ \cdots + n_k},X^k),$$
where $\ \simeq \ $ stands for the natural isomorphism, and the last arrow is composition of maps.\\

\begin{thm}{\em The sequence $\{\mathbb{Z}_2[\mathbb{A}^{n}]\}_{n\in \mathbb{N}}$ defines an operad equivalent to the endomorphism operad of $\mathbb{Z}_2$ in $\mathrm{set}$.}
\end{thm}

\begin{proof} The result follows from the identifications
$$\mathbb{Z}_2[\mathbb{A}^{n}] \ = \ \mathbb{Z}_2^{\mathbb{Z}_2^{n}} \ = \ \mathrm{M}(\mathbb{Z}_2^{n}, \mathbb{Z}_2) \ = \ \mathrm{End}_{\mathbb{Z}_2}(n).$$
\end{proof}

A $S$-collection is a sequence of sets $\ X=\{X_n\}_{n\in \mathbb{N}}\ $  such that $\ X_n\ $ comes with a right $S_n$-action.
A sequence of sets $\ A=\{A_n\}_{n\in \mathbb{N}}\ $ generates a $S$-collection with free $S_n$ actions, namely the sequence
$$\ A\times S\ = \ \{A_n\times S_n\}_{n\in \mathbb{N}}.$$   A $S$-collection $X$ generates the free operad $\ FX=  \{FX_n\}_{n\in \mathbb{N}}\ $ described in \cite{m}.
For an $S$-collection of the form $A\times S$, with $A$ any sequence of sets, the free operad  $$FA\ := \ F(A\times S)$$ generated by $A\times S$ admits the following description.  $FA_n$ is the set of all pairs $(t,\alpha)$ such that:

\begin{itemize}
  \item $t$ is a $A$-decorated planar rooted tree with $n$ marked incoming leaves and one outgoing vertex, namely the root. $A$-decorated means that a choice of an element in $A_k$ is made for each vertex in the tree with incoming valence $k$ (other than the marked incoming leaves.)

  \item A numbering of the marked incoming leaves, i.e. a bijection from the marked incoming leaves to $[n].$

  \item The action of $S_n$ on $FA_n$ permutes the numberings of the leaves.

  \item The operadic compositions is given by the grafting of trees.
\end{itemize}

Let $\mathbb{P}$ be the free operad in $\mathrm{Set}$  generated by the following sequence:
$$A_0= \{0,1\}, \ \ \ \ \ \ A_2  =  \{+,.\} \ \ \ \ \ \mbox{and} \ \ \ \ \  \ A_k \ = \ \emptyset \ \ \ \mbox{for} \ \ \ k\neq 0,2.$$

\begin{prop}{\em For $x \in \mathrm{set}$, the set of all propositions $\mathbb{P}(x)$ is  equal to
the free $\mathbb{P}$-algebra generated by $x$. }
\end{prop}

\begin{proof} By definition the free $\mathbb{P}$-algebra generated by $x$ is given by
$$\mathbb{P}(x) \ = \ \coprod_{n=0}^{\infty} \mathbb{P}(n)\times_{S_n} x^{ n},$$
where $\mathbb{P}(n)\times_{S_n} x^{n}$ is the quotient of $\mathbb{P}(n)\times x^{n}$ by the relations:
$$(f\alpha, a_1,...,a_n) \ \sim \ (f, a_{\alpha^{-1}1},...,a_{\alpha^{-1}n}) $$
for $f \in \mathbb{P}(n), \ \ \alpha \in S_n, \ \ \mbox{and} \ \ (a_1,...,a_n) \in x^n.$\\

 Since $\mathbb{P}$ is the free operad generated by the sequences of sets $A$, the free algebra can equivalently be described as the set of pairs $(t,f)$ where:
\begin{itemize}
  \item $t$ is a $A$-decorated planar rooted tree with $n$ marked incoming leaves. Thus actually $t$ is a binary tree with branches corresponding to $+,.,$ the sum a product symbols.
  \item $f$ is a map from the $n$ marked incoming leaves to $x$.
\end{itemize}
It is clear that such pairs $(t,f)$ are in bijective correspondence with proposition in $\mathbb{P}(x)$.
\end{proof}

We also denote by $\mathbb{P}$ the functor $$\mathbb{P}: \mathrm{set} \ \longrightarrow  \ \mathrm{Set}$$ sending $x$ to the the set of propositions $\ \mathbb{P}(x), \ $ and $\ f:x \longrightarrow y\ $ to
its unique extension $$\mathbb{P}(f): \mathbb{P}(x) \longrightarrow \mathbb{P}(y)$$ sending $x$ to $y$
via $f$, and respecting the logical connectives.\\

Let  $\mathrm{Req}$ be the category of equivalence relations.
Objects in $\mathrm{Req}$ are pairs $(X, R)$ where $X$ is a set and $R$
is an equivalence relation on $X$. A morphism $\ f: (X, R) \longrightarrow (Y, S)\ $ in $\ \mathrm{Req}\ $ is a map $f:X \longrightarrow Y$ such that
$fR \subseteq S.\ $ We have a functor $$\mathrm{Req} \ \longrightarrow \ \mathrm{Set}$$ sending $(X,R)$ to the quotient
set $X/R$. \\

\begin{prop}{\em
The pair $(\mathbb{P}, \mathbb{R})$ constructed above defines a functor $$(\mathbb{P}, \mathbb{R}): \mathrm{set} \ \longrightarrow \ \mathrm{Req}$$ sending $x\ $ to $\ (\mathbb{P}(x),\mathbb{R}(x)). \  $
Thus we obtain the quotient functor  $\mathbb{P}/\mathbb{R}: \mathrm{set} \ \longrightarrow \ \mathrm{set}.$}
\end{prop}

\begin{prop}\label{yc}{\em The functors $\mathbb{P}/\mathbb{R}\ $ and  $\ \mathbb{Z}_2[\mathbb{A}^{(\ )}]$ are naturally isomorphic, both as set valued functors and as Boolean algebras valued functors.}
\end{prop}

\begin{proof}It follows from Lemma \ref{bag} that $\ \mathbb{P}/\mathbb{R}(x)\ $ and $\ \mathbb{Z}_2[\mathbb{A}^{x}]\ $ are naturally isomorphic Boolean algebras.
\end{proof}

\begin{prop}{\em The operads $\ \{\mathbb{P}/\mathbb{R}[n]\}_{n\in \mathbb{N}}\ \ $ and  $\ \ \{\mathbb{Z}_2[\mathbb{A}^{n}]\}_{n\in \mathbb{N}}$.}
\end{prop}

\begin{proof} We know from Proposition \ref{yc} that $ \mathbb{P}/\mathbb{R}[n]$ and $\mathbb{Z}_2[\mathbb{A}^{n}]$ are isomorphic Boolean algebras.  The operadic operations on $\mathbb{P}/\mathbb{R}[n]$ arise from the operations of substitution and renaming of variables on propositional formulae. It is well-known that the formation of truth functions behaves well with respect to the operation of substitution and renaming of variables on propositional formulae. Thus $\ \{\mathbb{P}/\mathbb{R}[n]\}_{n\in \mathbb{N}}\ \ $ and  $\ \ \{\mathbb{Z}_2[\mathbb{A}^{n}]\}_{n\in \mathbb{N}}\ \ $ agree also as operads.

\end{proof}

Recall that   $\mathbb{Z}_2[\mathbb{A}^{x}]$
is a poset with $f\leq g \ $ if and only if $\ f(a)\leq g(a)$ for all $a \in \mathbb{Z}_2^x$. \\

Classical propositional logic main concern is the pre-order $\ \vdash \ $ of entailment on $\mathbb{P}(x)$.\\

The entailment relation $\ \vdash \ $ can be defined semantically as follows:
for $ p,q \in \mathbb{P}(x)$, set $$p \vdash q  \ \ \  \mbox{if and only if} \ \ \ \widehat{p} \leq \widehat{q},$$
or equivalently, $$p \vdash q \ \ \ \mbox{if and only if there is} \ \ r \in \mathbb{P}(x) \ \  \mbox{such that}
  \ \ \  \widehat{p}= \widehat{q}\widehat{r}. $$

The entailment relation $\vdash$ can be defined syntactically as follows:
$$p\vdash q \ \ \ \mbox{if and only if there exists} \ r \in  \mathbb{P}(x) \ \mbox{such that } \ p \sim qr. $$

Next we move from the  propositional settings  to the operational settings, always within  a Boolean context. Recall from the introduction that  quantum observables are operators instead of propositions. In analogy with the classical
case we identify operators with words in a certain language. On the semantic side truth functions are replaced by
Boolean differential operators. \\

We think of quantum operational logic as arising from the following principles:

\begin{enumerate}

\item Simple-composite operators. \\
On the syntactic side operators are words in a certain language. For a set $x$ we let $\widetilde{x}=\{\widetilde{a} \ | \ a \in x\}$ be a set disjoint
from $x$ whose elements are of the form $\ \widetilde{a}\ $ for $\ a \in x.$
Given $x$, the set  $\mathbb{O}(x)$ of operators is obtained from the set of simple operators
$x \sqcup \widetilde{x} \subseteq \mathbb{O}(x)$ using the binary connectives product $.$ and  sum  $+$, and the constants $0,1$. Explicitly,  $\mathbb{O}(x)$ is defined recursively as the set of words in the symbols
$a \in x \sqcup \widetilde{x},., +,0,1,(,)$  such that:
\begin{itemize}
\item $x \sqcup \widetilde{x} \subseteq \mathbb{O}(x),\ $ $0 \in \mathbb{O}(x), \ $ and $\ 1 \in \mathbb{O}(x).$
\item If $p,q \in \mathbb{O}(x),$ then $(pq)$ and $(p+q)$ are also in $\mathbb{O}(x)$.
\end{itemize}

\item The logical interpretation of the connectives $.$ and $+$.\\
The product $pq$ generalizes
the classical connective $\mathrm{AND}$, but there is also an ordering behind it: the operator $pq$ may be interpreted as  "act with the operator $q$,
$\mathrm{AND \ THEN}$ act with the operator $p.$"
The connective $+$ correspond to the exclusive or $\mathrm{XOR}$.
The constants $0$ and $1$ stand for the null operator and the identity operator, respectively. The logical interpretation is $\mathrm{RESET \ TO}\ 0 \ $ and $\ \mathrm{LEAVE \ AS \ IS}$, respectively.

\item The algebra $\mathrm{BDO}_x=\mathrm{End}_{\mathbb{Z}_2}(\mathbb{Z}_2[\mathbb{A}^x])$ of Boolean differential operators on $\mathbb{Z}_2^x.$\\
On the semantic side $\mathrm{BDO}_x$ is thought as the quantum analogue for the Boolean algebra $\mathbb{Z}_2[\mathbb{A}^x]$ of
truth functions. Just as we have a map from propositions to truth functions, we have a map $$\widehat{(\ )} : \mathbb{O}(x) \ \longrightarrow \ \mathrm{End}_{\mathbb{Z}_2}(\mathbb{Z}_2[\mathbb{A}^x])$$ from
operators to Boolean differential operators given by:
\begin{itemize}
\item $\widehat{a}$ is the operator of multiplication by $a$, for  $a \in x$.
\item $\widehat{\widetilde{a}}$ is  the Boolean partial derivative $\partial_a$ along $a$, for  $a \in x$.
\item $\widehat{p+q}=\widehat{p}+\widehat{q} \ $ and $\ \widehat{pq}=\widehat{p}\widehat{q} \ $ for $\ p,q \in \mathbb{O}(x).$
\item $\widehat{0}$ is the operator identically equal to $0$, and $\widehat{1}$ is the identity
operator.
\end{itemize}
\end{enumerate}

We think of the composition $\circ$ of operators in $\mathrm{End}_{\mathbb{Z}_2}(\mathbb{Z}_2[\mathbb{A}^x])$ as the quantum
analogue of the  meet $\wedge$, or equivalently  the product, on $\mathbb{Z}_2[\mathbb{A}^x] = \mathrm{M}(\mathbb{Z}_2^x, \mathbb{Z}_2)$.
Indeed $\circ$ is an extension of the classical meet. Consider the inclusion map $\mathbb{Z}_2[\mathbb{A}^x] \longrightarrow \mathrm{End}_{\mathbb{Z}_2}(\mathbb{Z}_2[\mathbb{A}^x])$
sending $f \in \mathbb{Z}_2[\mathbb{A}^x]$ to the operator of multiplication by $f$. This map is additive and multiplicative, thus showing
that the quantum structures are, as they should, an extension
of the classical ones. \\

The map $\mathbb{O}(x) \longrightarrow \mathrm{End}_{\mathbb{Z}_2}(\mathbb{Z}_2[\mathbb{A}^x])$ turns out to be surjective, and there is
a well-defined procedure to tell when two operators are assigned the same Boolean differential operator, which we proceed to introduce. Sub-words are defined in
$\mathbb{O}(x)$ just as in propositional logic. We define an equivalence relation $\mathbb{R}(x)$, also denoted by $\sim$, on $\mathbb{O}(x)$.
For $p,q \in \mathbb{O}(x)$ we set:
$$p \ \mathbb{R}(x) \ q \ \ \ \ \mbox{if and only if} \ \ \ \ \widehat{p}=\widehat{q}.$$
$ \mathbb{R}(x)$ is defined  syntactically as follows:  $p$ and $q$ are related if and only if either $p=q$ or there exists a sequence
$p_1,...,p_k$, for some $k \geq 1$,  such that $p_1=p$, and $p_k=q$, and $p_{i+1}$ is obtained from $p_i$ by replacing
a sub-word of $p_i$ by an equivalent word according to the following relations:

\begin{itemize}

\item Associativity for the product: $ \ p(qr) \ \sim \  (pq)r.$

\item Associativity and commutativity for  $+$:
 $$(p+q)+r \ \sim \ p+(q+r),\ \ \ \ \ \ \ \ p+q \ \sim \ q+p.$$
\item Distributivity: $\ p(q+r) \ \sim \ pq+pr.$
\item Additive and multiplicative units: $\ 0+ p \ \sim \ p\ $ and  $\ 1p \ \sim \  p$
\item Additive nilpotency: $\ p+p  \ \sim \ 0$.
\item Multiplicative idempotency and nilpotency: $\ aa \ \sim \ a\ \ $ and $\ \ \widetilde{a}\widetilde{a} \ \sim \ 0$, for $\ a \in x.$
\item Commutation relations: $$ ba \ \sim \ ab\ \ \ \ \ \mbox{and} \ \ \ \ \ \widetilde{a}\widetilde{b} \ \sim \ \widetilde{b}\widetilde{a}, \ \ \ \mbox{for} \ \ a,b \in x,$$
$$\widetilde{b}a \ \sim \ a\widetilde{b} \ \ \ \ \mbox{for} \ \ \ \ a \neq b \in x, \ \ \ \ \mbox{and}$$
$$\widetilde{a}a \ \sim \ a\widetilde{a} + \widetilde{a} +1 \ \ \ \ \mbox{for} \ \ \ \  a \in x.$$
\end{itemize}

Let $\mathbb{B}$ be the category of finite sets and bijections \cite{RE2}. We recall that an $S$-collection as defined above is the same as a functor from $\mathbb{B}\ $ to $\ \mathrm{Set}$, also known as  a combinatorial species following the Montreal school initiated by Joyal.\\

Note that we may indeed regard $\ \mathbb{O}, \ \mathbb{R}, \ $ and $\ \mathrm{End}_{\mathbb{Z}_2}(\mathbb{Z}_2[\mathbb{A}^{( \ )}]) \ $ as functors $$\mathbb{B}\ \longrightarrow \ \mathrm{Set}.$$
The pair $(\mathbb{O}, \mathbb{R})$  defines a functor $$(\mathbb{O}, \mathbb{R}):\mathbb{B} \ \longrightarrow \ \mathrm{Req}.$$

\begin{prop}\label{spc}{\em The functors $\ \mathbb{O}/ \mathbb{R}\ $ and $\ \mathrm{End}_{\mathbb{Z}_2}(\mathbb{Z}_2[\mathbb{A}^{( \ )}])\ $
are naturally isomorphic as Boole-Weyl algebra valued functors.}
\end{prop}

\begin{proof} It follows from Theorem \ref{qbas} that for any finite set $x$ we have natural isomorphisms
$$\mathbb{O}/\mathbb{R}(x) \ \simeq \ \mathrm{End}_{\mathbb{Z}_2}(\mathbb{Z}_2[\mathbb{A}^x])$$ respecting the structures of Boolean algebras on both sides.
\end{proof}

Next we define the entailment pre-order $\ \vdash \ $ on $\ \mathbb{O}(x)$.\\

First we define a   pre-order on   $\mathrm{End}_{\mathbb{Z}_2}(\mathbb{Z}_2[\mathbb{A}^x])$. Given
$S,T \in \mathrm{End}_{\mathbb{Z}_2}(\mathbb{Z}_2[\mathbb{A}^x])$ we set
$$S \leq T \ \ \ \mbox{if and only if there exists } R \ \mbox{such that}\ \ S=TR.$$ For example, if $\ S=\pi_A \ $ and $\ T=\pi_B \ $ are projections onto the
subspaces $A$ and $B$ of $\mathbb{Z}_2[\mathbb{A}^x]$, respectively, then
$\ \pi_A\leq \pi_B \ $ if and only if  $\ A\subseteq B, \ $ since $\ \pi_A=\pi_B R \ $ implies that
$$\pi_B\pi_A=\pi_B \pi_B R=\pi_B R= \pi_A, \ \ \ \ \mbox{and \ thus} \ \ \ \ A\subseteq B.$$

Entailment is defined semantically as follows, for $p,q \in \mathbb{O}(x)$ we set
$$p \vdash q \ \ \ \mbox{if and only if there is} \ \  r\in \mathbb{O}(x) \ \  \mbox{such that} \ \  \widehat{p}=\widehat{q}\widehat{r}.$$
Equivalently, the entailment relation $\ \vdash \ $ can be defined syntactically as:
$$p\vdash q\ \ \ \mbox{if and only if there is} \ \ r \in  \mathbb{O}(x) \ \ \mbox{such that} \ \ p \sim qr.$$

As expected,  entailment on operators is an extension of  entailment on propositions.

\begin{prop}{\em The $(\mathbb{P}(\ ),\vdash),\ $ $(\mathbb{O}(\ ), \vdash), \  (\mathbb{Z}_2[\mathbb{A}^{( \ )}], \leq ), \
(\mathrm{End}_{\mathbb{Z}_2}(\mathbb{Z}_2[\mathbb{A}^{( \ )}]), \leq ) $ may be regarded as functors from $\mathbb{B}$ to the category of pre-ordered sets. Moreover, these functors fit into the following commutative diagram of natural transformations:
\[\xymatrix @R=.3in  @C=.4in
{ (\mathbb{P}(\ ), \vdash) \ar[r] \ar[d] &  (\mathbb{O}(\ ), \vdash) \ar[d]
\\ (\mathbb{Z}_2[\mathbb{A}^{( \ )}], \leq ) \ar[r] & (\mathrm{End}_{\mathbb{Z}_2}(\mathbb{Z}_2[\mathbb{A}^{(\ )}]), \leq )  } \]
where the top horizontal arrow is the natural inclusion map, the bottom horizontal arrow sends $f$ to the operator of
multiplication by $f$, and the vertical arrows are the valuation maps
from propositions and operators to truth functions and Boolean differential operators, respectively.}
\end{prop}

Let $\mathbb{Z}_2$-$\mathrm{Vect}$ be the category of vector spaces and linear transformation over $\mathbb{Z}_2$.   We recall from \cite{markl} that a prop $P$ in $\mathbb{Z}_2$-$\mathrm{Vect}$ is a strict symmetric monoidal category such that:
\begin{itemize}
  \item The objects are the natural numbers $\mathbb{N}.$
  \item The monoidal structure is given on objects by $n \otimes m = m+n.$
  \item The symmetric monoidal structure is enriched over  $\mathbb{Z}_2$-$\mathrm{Vect}$. In particular, the set of morphisms $P(n,m)$ are $\mathbb{Z}_2$-vector spaces, and the product $f\otimes g$ on morphisms is bilinear.
\end{itemize}

Each vector space $V$ over $\mathbb{Z}_2$ determines the prop $\mathrm{End}_V$ of endomorphisms,  given by
$$\mathrm{End}_V(n,m) \ = \ \mathrm{Hom}_{\mathbb{Z}_2}(V^{\otimes n}, V^{\otimes m}). $$

As we saw in Proposition \ref{spc}, propositional logic give us a syntactical presentation of the endomorphism operad of $\mathbb{Z}_2$ in the category of sets.  Next we extend this result to quantum operational logic which give us
a syntactical presentation of the diagonal part of the endomorphism prop of
$\mathbb{Z}_2[\mathbb{A}^1]$ in the category $\mathbb{Z}_2$-$\mathrm{vect}$.\\

It is easy to check that if $P$ is a prop, then its diagonal part $P(n,n)$ is naturally an $S$-collection, with the right action
$P(n,n)\times S_n \ \longrightarrow \ P(n,n) $ given by $$f\alpha \ = \ \alpha \circ f \circ \alpha^{-1}.$$

\begin{thm}{\em The functor $\mathbb{O}/\mathbb{R}$ is isomorphic to the diagonal part of the prop
$\ \mathrm{End}_{\mathbb{Z}_2[\mathbb{A}^1]}$ as Boole-Weyl algebra valued functors.}
\end{thm}

\begin{proof}
 We have the following chain of natural isomorphism $$\mathbb{O}/\mathbb{R}[n] \simeq  \mathrm{End}_{\mathbb{Z}_2}(\mathbb{Z}_2[\mathbb{A}^{n}]) \simeq
\mathrm{End}_{\mathbb{Z}_2}(\mathbb{Z}_2[\mathbb{A}^1]^{\otimes n})  = \mathrm{Hom}_{\mathbb{Z}_2}(\mathbb{Z}_2[\mathbb{A}^1]^{\otimes n}, \mathbb{Z}_2[\mathbb{A}^1]^{\otimes n})
 =  \mathrm{End}_{\mathbb{Z}_2[\mathbb{A}^1]}(n,n)$$ respecting the structures of Boole-Weyl algebras, where the first isomorphim comes from Theorem \ref{qbas}, and the second isomorphism comes from the fact that fact that $\mathbb{Z}_2[\mathbb{A}^{n}]\simeq\mathbb{Z}_2[\mathbb{A}^1]^{\otimes n}. \ $  These isomorphims are $S_n$-equivariant since one can check for $\alpha \in S_n$ that
$$\alpha\circ x_i \circ \alpha^{-1} \ = \ x_{\alpha^{-1}i} \ \ \ \ \ \ \ \ \mbox{and} \ \ \ \ \ \ \ \  \alpha\circ y_i \circ \alpha^{-1}\ = \ y_{\alpha^{-1}i}.$$
\end{proof}

\section{Set Theoretical Viewpoint}\label{s6}

The link between classical propositional logic and the algebra of sets arises as follows. Recall
that there is a map $$\mathbb{P}(x) \ \longrightarrow \ \mathrm{M}(\mathbb{Z}_2^x, \mathbb{Z}_2)$$ sending each proposition
to its truth function. Since $\mathrm{M}(\mathbb{Z}_2^x, \mathbb{Z}_2)$ can be identified with $\mathrm{P}\mathrm{P}(x)$
we obtain a map $$ \mathbb{P}(x) \ \longrightarrow \ \mathrm{P}\mathrm{P}(x) $$ assigning to each proposition $p$  a set $\widehat{p}$ of subsets
of $x$. Moreover, the logical connectives intertwine nicely with the set theoretical operations on subsets, namely:
$$ \widehat{p + q} = (\widehat{p} \cup \widehat{q}) \setminus  (\widehat{p} \cap \widehat{q}), \ \ \ \
\widehat{pq}=\widehat{p \wedge q} = \widehat{p} \cap \widehat{q}, \ \ \ \  \widehat{ \neg p } = \overline{\widehat{p} }, \ \ \ \  \widehat{p \vee q} = \widehat{p} \cup \widehat{q} \ \ \
\widehat{p \rightarrow q} = \overline{\widehat{p}} \cup \widehat{q}.$$
We stress the, often overlooked, fact that classical propositional logic describes the set theoretical
operations present in $\mathrm{P}\mathrm{P}(x)$ that are common to all sets of the form $\mathrm{P}(y),$ i.e.
the extra algebraic structures present in $\mathrm{P}(y)$ when $y=\mathrm{P}(x)$ play no significative role
in the logic/set theory relation outlined above. Thus whereas the axioms characterizing the algebras $\mathrm{P}(x)$ have been massively studied, the algebraic structures characterizing $\mathrm{P}^n(x)$, for $n \geq 2$, have seldom attracted any attention.\\

We proceed to consider the analogue statements in the quantum operational scenario.
We present our results for sets of the form $[n]$. It should be clear, however,
that the same constructions apply for arbitrary finite sets.\\

As in the classical case we have a map $\ \mathbb{O}_n \ \longrightarrow \ \mathrm{End}_{\mathbb{Z}_2}(\mathbb{Z}_2[\mathbb{A}^n]) \ $ sending operators (words in a certain language) to
Boolean differential operators. As shown in Section \ref{bwa} it is possible to identify $\mathrm{End}_{\mathbb{Z}_2}(\mathbb{Z}_2[\mathbb{A}^n])$ with the Boolean-Weyl algebra
$\mathrm{BW}_n$, and with the shifted Boolean-Weyl algebra $\mathrm{SBW}_n$.
Moreover, we described several explicit bases for these algebras.
For example, each $f\in \mathrm{BW}_n$ can be written in an unique way as:
$$f= \sum_{a, b \in \mathrm{P}[n]}f(a,b)x^ay^b.$$
Thus Boolean differential operators can be identified with maps from $\mathrm{P}[n] \times \mathrm{P}[n]$
to $\mathbb{Z}_2$, and we get the identifications:
$$\mathrm{End}_{\mathbb{Z}_2}(\mathbb{Z}_2[\mathbb{A}^n]) \simeq \mathrm{BDO}_n \simeq \mathrm{BW}_n \simeq \mathrm{M}(\mathrm{P}[n] \times \mathrm{P}[n], \mathbb{Z}_2) \simeq
 \mathrm{P}(\mathrm{P}[n] \times \mathrm{P}[n]) \simeq  \mathrm{P}\mathrm{P}([n] \sqcup [n]).$$

We adopt the following conventions. We identify $[n] \sqcup [n]$ with the set
$$[n,\widetilde{n}]=\{1,2,...,n, \widetilde{1}, \widetilde{2},..., \widetilde{n} \}.$$
Given $a \subseteq [n]$ we let $\widetilde{a}=\{\widetilde{i} \ | \ i \in a\}$ be the corresponding subset of $[\widetilde{n}]=\{\widetilde{1}, \widetilde{2},..., \widetilde{n} \}.$ An element $a \in \mathrm{P}[n,\widetilde{n}]$ will be written as
$a=a_1 \sqcup \widetilde{a}_2$ with $a_1, a_2 \in \mathrm{P}[n]$. Note that we have a natural map
$\pi :  \mathrm{P}[n,\widetilde{n}] \longrightarrow \mathrm{P}[n]  \times \mathrm{P}[n]$ given
by $\pi (a)=(\pi_1(a), \pi_2(a))=(a_1, a_2)$.
We use indices without tilde to denote monomials of regular functions, and indices
with tilde to denote the Boolean derivatives or shift operators.
The identification $\mathrm{End}_{\mathbb{Z}_2}(\mathbb{Z}_2[\mathbb{A}^n]) = \mathrm{P}\mathrm{P}[n,\widetilde{n}]$ allow us to give
set theoretical interpretations to the algebraic structures on Boolean differential
operators. Unlike the classical set theoretical structures, the quantum operational structures are not defined for
arbitrary sets of the form $\mathrm{P}(y)$, quite to the contrary, they very much depend on the fact that  $y=\mathrm{P}[n,\widetilde{n}].$\\

Below we consider pairs $(A,M)$ where $A$ is a  $\mathbb{Z}_2$-algebra and $M$ is an $A$-module. A morphism
$(f,g):(A_1,M_1) \rightarrow (A_2,M_2)$ between such pairs, is given by  $\mathbb{Z}_2$-linear maps
$f:A_1  \rightarrow A_2$  and $g: M_1 \rightarrow M_2$ such that $f$ is an algebra morphism, and
$g(am)=f(a)g(m)$ for all $a \in A, m \in M.$\\

The additive structure $+: \mathrm{P}\mathrm{P}[n,\widetilde{n}] \times  \mathrm{P}\mathrm{P}[n,\widetilde{n}] \ \longrightarrow \
 \mathrm{P}\mathrm{P}[n,\widetilde{n}]\ $ on $\ \mathrm{P}\mathrm{P}[n,\widetilde{n}]\ $
is given by $$A+B\ = \ A \cup B \setminus (A \cap B).$$

We consider several isomorphic products $\ \circ, \ \bullet, \ \star, \ $ and  $\ \ast \ $ on $\mathrm{P}\mathrm{P}[n,\widetilde{n}]$
displaying different combinatorial properties. The products correspond with the various bases
for $\mathrm{BW}_n$ and $\mathrm{SBW}_n$.\\

We first introduce the product $\circ$  on $\mathrm{P}\mathrm{P}[n,\widetilde{n}].$

\begin{thm}\label{lala}{\em There are maps $$\circ: \mathrm{P}\mathrm{P}[n,\widetilde{n}] \times \mathrm{P}\mathrm{P}[n,\widetilde{n}]\  \longrightarrow \ \mathrm{P}\mathrm{P}[n,\widetilde{n}] \ \ \ \ \ \mbox{and} \ \ \ \ \ \circ:\mathrm{P}\mathrm{P}[n,\widetilde{n}] \times \mathrm{P}\mathrm{P}[n]\ \longrightarrow \ \mathrm{P}\mathrm{P}[n],$$ turning $\mathrm{P}\mathrm{P}[n,\widetilde{n}]$ into a $\mathbb{Z}_2$-algebra and
$\mathrm{P}\mathrm{P}[n]$ into a module over $\mathrm{P}\mathrm{P}[n,\widetilde{n}],$ such that  the pair
$(\mathrm{P}\mathrm{P}[n,\widetilde{n}], \mathrm{P}\mathrm{P}[n])$ is isomorphic to
$\left( \mathrm{End}_{\mathbb{Z}_2}( \mathrm{M}(\mathbb{Z}_2^n, \mathbb{Z}_2) ),   \mathrm{M}(\mathbb{Z}_2^n, \mathbb{Z}_2 ) \right)$
via the maps $$A \ \longrightarrow \ \sum_{a \in A}m^{a_1}\partial^{a_2} \ \ \ \ \ \ \ \mbox{and} \ \ \ \ \ \ \ F \ \longrightarrow \  \sum_{a \in F}m^{a}.$$}
\end{thm}

\begin{proof}
From Theorem \ref{tt1} and Proposition \ref{do} we see that the desired products $\circ$ are constructed as follows.
For $A,B  \in \mathrm{P}\mathrm{P}[n,\widetilde{n}],$  the product $AB \in \mathrm{P}\mathrm{P}[n,\widetilde{n}]$ is given by
$$A\circ B \ = \ \left\{ a \in \mathrm{P}[n,\widetilde{n}] \ \Big| \ O\Big\{ b \in \mathrm{P}[n],c \in B \ | \
c_2 \subseteq a_2, \ a_1 \sqcup \widetilde{b} \in A, \ a_2 \setminus c_2 \subseteq a_1 +c_1 \subseteq b \Big\} \right\}.$$ Let $A \in \mathrm{P}\mathrm{P}[n,\widetilde{n}] $ and $F \in \mathrm{P}\mathrm{P}[n]$, then
$AF \in \mathrm{P}\mathrm{P}[n]$ is given by
$$A\circ F \ = \ \Big\{ \ a \in \mathrm{P}[n] \ \ | \ \ O\{b \subseteq c \in \mathrm{P}[n] \ | \ a \sqcup \widetilde{c} \in A,  \
a + b \in F \} \ \Big\} .$$
\end{proof}

We provide a few applications of Theorem \ref{lala}.

\begin{exmp}{\em  In $\mathrm{P}\mathrm{P}[3,\widetilde{3}]$ we have that $$\{\{1,2,\overline{2},\overline{3} \}\}\circ \{\{1,3,\overline{1},\overline{2} \}\}\ = \
\{ \{1,2,\overline{1},\overline{2} \}, \ \{1,2,\overline{1},\overline{2},\overline{3}\}  \}.$$ Indeed
$a \in \{\{1,2,\overline{2},\overline{3} \}\}\circ \{\{1,3,\overline{1},\overline{2} \}\}$ if there is a odd number
of suitable pairs $b,c$. Note that $$c = \{1,3,\overline{1},\overline{2} \}, \ \ \ \ \ \{1,2\}\subseteq a_2, \ \ \ \ \
a_1 \sqcup \widetilde{b}=\{1,2,\overline{2},\overline{3}\},$$ and thus necessarily $a_1=\{1,2\}$ and $b=\{2,3\}$.
Moreover, we must have that $$a_2 \setminus \{1,2\} \subseteq \{1,2\} + \{1,3\} \subseteq \{2,3\}, \ \ \ \ \ \mbox{that is}
\ \ \ \ \ a_2 \setminus \{1,2\} \subseteq \{2,3\}.$$ Thus either $a_2=\{1,2\}$ or $a_2=\{1,2,3\}$ yielding the desired result.

 }
\end{exmp}

\begin{exmp}{\em For $A \in \mathrm{P}\mathrm{P}[n]$ set $A'=\pi_2^{-1}(A)$. Let $F \in \mathrm{P}\mathrm{P}[n]$, the we have that:
$$A'\circ F\ = \ \Big\{ \ a \in \mathrm{P}[n] \ \ | \ \ O\{b \subseteq c \in \mathrm{P}[n] \ | \  \widetilde{c} \in A,  \
a + b \in F \} \ \Big\} .$$
Note that $\ \sum_{a \in \mathrm{P}[n]}m^a=1\ $ and thus:
$$\left( \sum_{a \in A} \partial^a \right)\circ \left( \sum_{b \in F} m^b \right)\ = \
 \sum_{a \in  A'\circ F} m^a .$$
 }
\end{exmp}

\begin{exmp}{\em For $A  \in  \mathrm{P}\mathrm{P}[n]\ $ let $\ \widehat{A} =\{a \in \mathrm{P}[n,\widetilde{n}] \ | \ a_1 =a_2 \in A \ \}$. Let $F \in  \mathrm{P}\mathrm{P}[n]$, then
$$\widehat{A}\circ F\ = \ \left\{ \ a \in \mathrm{P}[n] \ \ | \ \ O\{e \subseteq a \ | \  \ a+e \in F \} \ \right\} \ \  \mbox{and therefore}$$
$$\left( \underset{{a \in A}} {\sum}m^a \partial^a  \right)\circ \left(\underset{{b \in F}} {\sum}m^b\right)\ = \
\sum_{a \in \widehat{A}\circ F }m^a.$$ }
\end{exmp}

\

Next we introduce the product $\bullet$  on $\mathrm{P}\mathrm{P}[n,\widetilde{n}].$

\begin{thm}\label{lolo}{\em There are maps $$\bullet : \mathrm{P}\mathrm{P}[n,\widetilde{n}] \times \mathrm{P}\mathrm{P}[n,\widetilde{n}]\  \longrightarrow \ \mathrm{P}\mathrm{P}[n,\widetilde{n}] \ \ \ \ \ \mbox{and} \ \ \ \ \ \bullet : \mathrm{P}\mathrm{P}[n,\widetilde{n}] \times \mathrm{P}\mathrm{P}[n]\ \longrightarrow \ \mathrm{P}\mathrm{P}[n]$$ such that the pair
$(\mathrm{P}\mathrm{P}[n,\widetilde{n}], \mathrm{P}\mathrm{P}[n])$ is isomorphic to
$\left( \mathrm{End}_{\mathbb{Z}_2}( \mathrm{M}(\mathbb{Z}_2^n, \mathbb{Z}_2) ),   \mathrm{M}(\mathbb{Z}_2^n, \mathbb{Z}_2 ) \right)$
via the maps $$A \ \longrightarrow \ \sum_{a \in A}x^{a_1}\partial^{a_2} \ \ \ \ \ \mbox{and} \ \ \ \ \ F \ \longrightarrow \ \sum_{a \in F}x^{a}.$$}
\end{thm}

\begin{proof}
From Theorem \ref{tt1} and Proposition \ref{do} we see that the desired products $\bullet$ are constructed as follows.
For $A,B  \in \mathrm{P}\mathrm{P}[n,\widetilde{n}],$  the product $A\bullet B \in \mathrm{P}\mathrm{P}[n,\widetilde{n}]$ is such that
$a \in A\bullet B$ iff
$$O\Big\{b \in A, c \in B, k_1 \subseteq k_2 \ \Big|
\ b_1 \subseteq a_1,\ c_1 \subseteq a_2, \ k_2 \subseteq b_2 \cap c_1, \ b_1 \cup ( c_1 \setminus k_2) = a_1, \  b_2 \setminus k_1= a_2 \setminus c_2 \Big\}.$$ Let $\ A \in \mathrm{P}\mathrm{P}[n,\widetilde{n}] \ $ and $\ F \in \mathrm{P}\mathrm{P}[n], \ $ then $\ A \bullet F \in \mathrm{P}\mathrm{P}[n]\ $ is given by
$$A \bullet F \ = \ \Big\{ \ a \in \mathrm{P}[n] \ \ | \ \
O\{ b \in A, c \in F \ \Big| \ b_2 \subseteq c, \  b_1 \cup (c \setminus b_2)=a\} \ \Big\} .$$
\end{proof}

We provide a few applications of Theorem \ref{lolo}.

\begin{exmp}{\em In $\mathrm{P}\mathrm{P}[3,\widetilde{3}]$ we have that $$\{\{1,3,\overline{2} \}\} \bullet \{ \{2,\overline{1} \}\}\ = \
\{\{1,2,3, \overline{1}, \overline{2} \} , \{1,3,\overline{1},\overline{2} \} , \{1,3,\overline{1} \}\}.$$
Indeed, we must have $b=\{1,3,\overline{2} \}\ $ and $\ c= \{2,\overline{1} \}$, and thus
there are three options for $$k_1 \subseteq k_2 \subseteq [2], \ \ \ \mbox{namely} \ \ \ \emptyset \subseteq \emptyset, \ \ \
\emptyset \subseteq \{2\}, \ \ \ \mbox{and} \ \ \ \{2\} \subseteq \{2\},$$ giving rise to the sets
$\{1,2,3, \overline{1}, \overline{2} \} , \ \ \{1,3,\overline{1},\overline{2} \} , \ \ \{1,3,\overline{1} \},$
respectively.

}
\end{exmp}

\begin{exmp}{\em For $A\in \mathrm{P}\mathrm{P}[n]\ $ set $\ A'=\pi^{-1}(\{\emptyset\}\times A). \ $ Let $F \in \mathrm{P}\mathrm{P}[n]$ then we have that:
$$A'\bullet F\ = \ \Big\{ \ a \in \mathrm{P}[n] \ \ \Big| \ \
O\{b \in A, c \in F \ | \ b \subseteq c, \  c \setminus b=a\} \ \Big\}$$
Therefore we get that
$$\left( \sum_{a \in A} \partial^a \right)\circ \left( \sum_{b \in F} x^b \right)\ = \ \sum_{a \in A'\bullet F} x^a. $$
}
\end{exmp}

\begin{exmp}{\em  For $A  \in  \mathrm{P}\mathrm{P}[n]$ let
$\widehat{A} =\{a \in \mathrm{P}[n,\widetilde{n}] \ | \ a_1 =a_2 \in A \ \}$. Let $F \in  \mathrm{P}\mathrm{P}[n]$, then
$$\widehat{A}\bullet F\ = \ \left\{ \ a \in F \ \ | \ \ O\{b \in A \ | \  \ b \subseteq a \} \ \right\} \ \  \mbox{and therefore}$$
$$\left( \underset{{a \in A}} {\sum}x^a \partial^a  \right)\circ \left(\underset{{b \in F}} {\sum}x^b\right) \ = \
\sum_{a \in \widehat{A}\bullet F }x^a \ = \ \sum_{a \in F}O\{ b \in A \ | \ b \subseteq a \}x^a.$$
In particular we have that:
$\ \ \widehat{\mathrm{P}[n]}\bullet \mathrm{P}[n]= \left\{ \emptyset \right\}\ $ and thus
$$\left( \underset{{a \in \mathrm{P}[n]}} {\sum}x^a \partial^a  \right)\circ \left(\underset{{b \in \mathrm{P}[n]}} {\sum}x^b\right)\ = \ 1.$$
}
\end{exmp}

\

Next we introduce the product $\star$  on $\mathrm{P}\mathrm{P}[n,\widetilde{n}].$

\begin{thm}\label{lulu}{\em There are maps $$\star: \mathrm{P}\mathrm{P}[n,\widetilde{n}] \times \mathrm{P}\mathrm{P}[n,\widetilde{n}] \ \longrightarrow  \ \mathrm{P}\mathrm{P}[n,\widetilde{n}] \ \ \ \ \ \ \ \mbox{and} \ \ \ \ \ \ \ \star:\mathrm{P}\mathrm{P}[n,\widetilde{n}] \times \mathrm{P}\mathrm{P}[n] \ \longrightarrow \ \mathrm{P}\mathrm{P}[n],$$ turning $\mathrm{P}\mathrm{P}[n,\widetilde{n}]$ into a $\mathbb{Z}_2$-algebra and
$\mathrm{P}\mathrm{P}[n]$ into a module over $\mathrm{P}\mathrm{P}[n,\widetilde{n}],$ such that  the pair
$(\mathrm{P}\mathrm{P}[n,\widetilde{n}], \mathrm{P}\mathrm{P}[n])$ is isomorphic to
$\left( \mathrm{End}_{\mathbb{Z}_2}( \mathrm{M}(\mathbb{Z}_2^n, \mathbb{Z}_2) ),   \mathrm{M}(\mathbb{Z}_2^n, \mathbb{Z}_2 ) \right)$
via the maps $$A \ \longrightarrow \ \sum_{a \in A}m^{a_1}s^{a_2} \ \ \ \ \ \ \ \mbox{and} \ \ \ \ \ \ \ F \ \longrightarrow \  \sum_{a \in F}m^{a}.$$}
\end{thm}

\begin{proof}
From Theorem \ref{tt2} and Proposition \ref{dos} we see that the desired products $\star$ are constructed as follows.
For $A,B  \in \mathrm{P}\mathrm{P}[n,\widetilde{n}],$  the product $A\star B \in \mathrm{P}\mathrm{P}[n,\widetilde{n}]$ is given by
$$A\star B \ = \ \left\{ a \in \mathrm{P}[n,\widetilde{n}] \ | \ O\{ b \in \mathrm{P}[n] \ | \
a_1 \sqcup \widetilde{b} \in A, \ (a_1 + b )\sqcup \widetilde{(a_2 + b)} \in B \} \right\}.$$
Let $A \in \mathrm{P}\mathrm{P}[n,\widetilde{n}] $ and $F \in \mathrm{P}\mathrm{P}[n]$, then
$A\star F \in \mathrm{P}\mathrm{P}[n]$ is given by
$$A\star F \ = \ \left\{  a \in \mathrm{P}[n] \ \ | \ \ O\{b  \in \mathrm{P}[n] \ | \ a \sqcup \widetilde{b} \in A,  \
a + b \in F \}  \right\} .$$
\end{proof}

We provide a few applications of Theorem \ref{lulu}.

\begin{exmp}{\em In $\mathrm{P}\mathrm{P}[3,\widetilde{3}]$ we have that $\{\{1,2,3,\overline{3} \}\} \star \{ \{1,2,\overline{2},\overline{3} \}\}=
\{\{1,2,3, \overline{2} \} \}.$
From the equation $a_1 \sqcup \widetilde{b} \in A$ we see that $a_1=\{1,2,3\}$ and
$b=\{3 \}$. Also we must have $a_1 + \{3 \}=\{1,2\}$, which holds, and $a_2 + \{3 \}=\{2,3\}$ which implies that
$a_2=\{2 \}.$
}
\end{exmp}

\begin{exmp}{\em For $A \in \mathrm{P}\mathrm{P}[n]$ set $A'=\pi_2^{-1}(A)$. Let $F \in \mathrm{P}\mathrm{P}[n]$, the we have that:
$$A'\star F \ = \ \Big\{  a \in \mathrm{P}[n] \ \ \Big| \ \
O\{b  \in A \ | \  a + b \in F \} \ \Big\}.$$
Therefore
$$\left( \sum_{a \in A} s^a \right)\circ \left( \sum_{b \in F} m^b \right) = \sum_{a \in A'\star F} m^a \ \
\mbox{in particular} \ \ \left( \sum_{a \in A} s^a \right)\circ \left( \sum_{b \in \mathrm{P}[n]} m^b \right)=OA\sum_{a \in  \mathrm{P}[n]} m^a .$$
}
\end{exmp}

\begin{exmp}\label{53}{\em  For $A  \in  \mathrm{P}\mathrm{P}[n]$ set
$\widehat{A} =\{a \in \mathrm{P}[n,\widetilde{n}] \ | \ a_1 =a_2 \in A \ \}$. Let $F \in  \mathrm{P}\mathrm{P}[n]$, then
$\widehat{A}\star F= \emptyset \ \ \mbox{if} \ \ \emptyset \not \in F \  \ \mbox{and} \ \
\widehat{A}\star F=A\ \ \mbox{if} \ \ \emptyset \in F. \ $ Therefore
$$\left( \underset{{a \in A}} {\sum}m^a s^a  \right)\circ \left(\underset{{b \in F}} {\sum}m^b\right)\ = \
c\sum_{a \in A }m^a,$$
where $c=1\ $ if $\ \emptyset \in F, \ \ $ and $\ \ c=0\ $ if $\ \emptyset \not \in F$.}
\end{exmp}

\begin{exmp}{\em Let $\widehat{A}$ be as Example \ref{53}, then $\widehat{A} \star \widehat{A}=\widehat{A}$ if
$\emptyset \in A,\ $ and $\ \widehat{A} \star \widehat{A}=\emptyset$ otherwise. Indeed, $a \in \mathrm{P}[n,\widetilde{n}]$ belongs
to $\widehat{A} \star \widehat{A}$ if $a_1 \sqcup \widetilde{b} \in \widehat{A}$, i.e. $a_1=b \in A,$ and
$(a_1 + b )\sqcup \widetilde{(a_2 + b)} \in \widehat{A}$, i.e.
$\emptyset \in A$ and $a_1=a_2 \in A.$
}
\end{exmp}

\begin{exmp}{\em  For $A  \in  \mathrm{P}\mathrm{P}[n]$ set
$$\widetilde{A} =\{a \in \mathrm{P}[n,\widetilde{n}] \ | \ \overline{a_1} =a_2 \in A \ \}.$$
Then $\widetilde{A} \star \widetilde{A}=\widetilde{A}\ $ if
$ \ [n] \in A, \ \ $ and $\ \ \widehat{A} \star \widehat{A}=\emptyset\ $ if $\ [n] \not \in A$. Indeed, $a \in \mathrm{P}[n,\widetilde{n}]$ belongs
to $\widetilde{A} \star \widetilde{A}$ if $a_1 \sqcup \widetilde{b} \in \widetilde{A}$, i.e. $\overline{a}_1=b \in A,$ and
$(a_1 + b) \sqcup \widetilde{(a_2 + b)} \in \widehat{A}$,
i.e. $[n] \in A$ and $a_2=\emptyset+b=b= \overline{a}_1.$
}
\end{exmp}

\

Finally we introduce the product $\ast$  on $\mathrm{P}\mathrm{P}[n,\widetilde{n}].$

\begin{thm}\label{lili}{\em There are maps $$\ast: \mathrm{P}\mathrm{P}[n,\widetilde{n}] \times \mathrm{P}\mathrm{P}[n,\widetilde{n}]  \ \longrightarrow \ \mathrm{P}\mathrm{P}[n,\widetilde{n}] \ \ \ \ \ \ \ \mbox{and} \ \ \ \ \ \ \ \ast:\mathrm{P}\mathrm{P}[n,\widetilde{n}] \times \mathrm{P}\mathrm{P}[n] \ \longrightarrow \ \mathrm{P}\mathrm{P}[n],$$ turning $\mathrm{P}\mathrm{P}[n,\widetilde{n}]$ into a $\mathbb{Z}_2$-algebra and
$\mathrm{P}\mathrm{P}[n]$ into a module over $\mathrm{P}\mathrm{P}[n,\widetilde{n}],$ such that  the pair
$(\mathrm{P}\mathrm{P}[n,\widetilde{n}], \mathrm{P}\mathrm{P}[n])$ is isomorphic to
$\left( \mathrm{End}_{\mathbb{Z}_2}( \mathrm{M}(\mathbb{Z}_2^n, \mathbb{Z}_2) ),   \mathrm{M}(\mathbb{Z}_2^n, \mathbb{Z}_2 ) \right)$
via the maps $$A \ \longrightarrow \ \sum_{a \in A}x^{a_1}s^{a_2} \ \ \ \ \ \ \ \mbox{and} \ \ \ \ \ \ \ F \ \longrightarrow \  \sum_{a \in F}x^{a}.$$}
\end{thm}

\begin{proof}
From Theorem \ref{tt2} and Proposition \ref{dos} we see that the desired products $\star$ are constructed as follows.
For $A,B  \in \mathrm{P}\mathrm{P}[n,\widetilde{n}],$  the product $A\ast B \in \mathrm{P}\mathrm{P}[n,\widetilde{n}]$ is given by
$$\left\{ a \in \mathrm{P}[n,\widetilde{n}] \ \Big| \ O\{ b,c,d,e \in \mathrm{P}[n] \ | \
e \subseteq c\cap d,\ b\cup d\setminus e= a_1,\ b \sqcup \widetilde{c} \in A, \ d \sqcup \widetilde{(c+a_2)} \in B \} \right\}.$$
Let $A \in \mathrm{P}\mathrm{P}[n,\widetilde{n}] $ and $F \in \mathrm{P}\mathrm{P}[n]$, then
$A\ast F \in \mathrm{P}\mathrm{P}[n]$ is given by
$$A\ast F \ = \ \left\{ \ a \in \mathrm{P}[n] \ \ | \ \ O\{b,c,d \in \mathrm{P}[n], e \in F \ | \
c\subseteq d\cap e, \ b\cup e\setminus c= a, \ b \sqcup \widetilde{d} \in A \ \} \ \right\} .$$
\end{proof}

We provide a few applications of Theorem \ref{lili}.

\begin{exmp}{\em In $\mathrm{P}\mathrm{P}[3,\widetilde{3}]$ we have that $\{\{1,\overline{2} \}\} \ast \{ \{2,3,\overline{1}, \overline{2}\}\}=
\{ \{1,3,\overline{1}\} , \{1,2,3,\overline{1} \}\}.$ Indeed we must have $b=\{ 1\}$, $c=\{ 2\}$,
$d=\{ 2,3\}$, and $a_2= \{ 2\} + \{ 1,2\}=\{ 1\}$. Since $e \subseteq \{ 2\} \cap \{ 2,3\}= \{ 2\}$,
there are two options, either $e= \emptyset$ and then $a_1=\{ 1,2,3 \}$ and $a=\{ 1,2,3,\overline{1} \}$, or  $e = \{ 2\}$ and then
$a_1=\{ 1,3 \}$ and $a=\{ 1,3, \overline{1}\}.$
}
\end{exmp}

\begin{exmp}{\em For $A\in \mathrm{P}\mathrm{P}[n]$ set $A'=\pi^{-1}(\{\emptyset\}\times A)$. Let $F \in \mathrm{P}\mathrm{P}[n]$ then we have that:
$$A'\ast F \ = \ \Big\{ \ a \in \mathrm{P}[n] \ \ | \ \
O\{c \in \mathrm{P}[n], d \in A, e \in F \ | \ c \subseteq d \cap e, \  e\setminus c=a\} \ \Big\}.$$
Therefore
$$\left( \sum_{a \in A} s^a \right)\circ \left( \sum_{b \in F} x^b \right)\ = \ \sum_{a \in A'\ast F} x^a. $$
}
\end{exmp}

\begin{exmp}{\em  For $A  \in  \mathrm{P}\mathrm{P}[n]$ let
$\widehat{A} =\{a \in \mathrm{P}[n,\widetilde{n}] \ | \ a_1 =a_2 \in A \ \}$. Let $F \in  \mathrm{P}\mathrm{P}[n]$, then
$$\widehat{A}\ast F \ = \ \Big\{ \ a \in \mathrm{P}[n] \ \ \Big| \ \ O\{b\in A, c \in \mathrm{P}[n], e \in F \ | \
c\subseteq b\cap e, \ b\cup e\setminus c= a \ \} \ \Big\}.$$
$$\left( \sum_{a \in A} x^as^a \right)\circ \left( \sum_{b \in F} x^b \right)\ = \ \sum_{a \in \widehat{A}\ast F} x^a. $$
}
\end{exmp}

\section{Final Remarks}\label{s7}

We introduced an approach for the study the analogue of quantum-like structures in characteristic $2$. Our approach  is developed as follows: \\

\noindent 1) Quantization of the canonical phase space $k^n \times k^n$ over a field $k$  of characteristic zero may be identified with the
$k$-algebra of algebraic differential operators on $k^n.$ The Weyl algebra provides an explicit description by generators and relations
of the latter algebra.\\

\noindent 2) Classical propositional logic may be identified, to a good extend, with the study of regular
functions on $\mathbb{Z}_2^n$. We introduced  the algebra $\mathrm{BDO}_n$ as a
suitable analogue for algebraic differential operators on $\mathbb{Z}_2^n$.
We showed that $\mathrm{BDO}_n$ coincides with the algebra of linear endomorphisms of regular functions on $\mathbb{Z}_2^n$.
We introduced a couple of presentations by generators and relations for $\mathrm{BDO}_n$, namely, the quantum Boolean algebras
$\mathrm{BW}_n$ and $\mathrm{SBW}_n.$\\

\noindent 3) We shifted back from the algebro-geometric viewpoint, and study the quantum Boolean algebras
from the logical and set theoretical viewpoints.\\

Our work leaves several open questions and problems for future research:\\

\noindent 1) We considered the structural aspects of quantization in characteristic $2$. The
dynamical aspects  will be considered elsewhere.\\

\noindent 2) We studied the analogue for the Weyl algebra in characteristic $2.$ Recently,  \cite{co, deit, so},
there have been a remarkable interest in  characteristic $1$. Is there an analogue for the Weyl algebra in characteristic $1$.\\

\noindent 3) Categorification of the Weyl algebra has been considered in \cite{RE2}; categorification
of the Boole-Weyl  and shifted Boole-Weyl algebras remain to be addressed.\\

\noindent 4) The symmetric powers of Weyl algebras and linear Boolean algebras in characteristic zero were studied in  \cite{RE} and \cite{DR},
respectively. The analogue problems for the quantum Boolean algebras and the linear quantum
Boolean algebras are open.\\

\noindent 5) Our logical interpretation of quantum Boolean algebras was based on a specific choice of
connectives. It remains to study other connectives, perhaps with a more direct logical meaning. \\

\bigskip

\noindent ragadiaz@gmail.com\\
\noindent Instituto de Matem\'aticas y sus Aplicaciones, Universidad Sergio Arboleda, Bogot\'a, Colombia\\

\end{document}